\documentclass[12pt]{amsart}
\usepackage[left=3.5cm,right=3.5cm,top=3.4cm, bottom=3cm ]{geometry}

\usepackage{amstext}
\usepackage{amsthm}
\usepackage{amsmath}
\usepackage{amssymb}
\usepackage{latexsym}
\usepackage{amsfonts}
\usepackage{color}
\usepackage{graphicx}
\usepackage{yfonts}
\usepackage{verbatim}
\usepackage{enumitem}
\newcommand{\vertiii}[1]{{\left\vert\kern-0.25ex\left\vert\kern-0.25ex\left\vert #1
    \right\vert\kern-0.25ex\right\vert\kern-0.25ex\right\vert}}

\AtBeginDocument{%
 \def\MR#1{} 
}

\usepackage[pagebackref,hypertexnames=false, colorlinks, citecolor=black, linkcolor=black, urlcolor=red]{hyperref}
\usepackage[backrefs]{amsrefs}

\newcommand\black{\color{black}}

\usepackage{latexsym}
\usepackage{amssymb}
\usepackage{euscript}

\let\cal=\mathcal      

\def\mcc{M\raise.5ex\hbox{c}C}
\def\mccarthy{M\raise.5ex\hbox{c}Carthy}

\def\eg{{\it e.g. }}
\def\ie{{\it i.e. }}

\def\M{{\cal M}}

\def\z{\zeta}

\def\vare{\varepsilon}

\let\i=\infty

\def\la{\langle}
\def\ra{\rangle}
\def\={\ = \ }

\def\I{{\cal I}}

\def\C{\mathbb C}
\def\R{\mathbb R}

\def\D{\mathbb D}
\def\B{\mathbb B}

\def\be{\setcounter{equation}{\value{theorem}} \begin{equation}}
\def\ee{\end{equation} \addtocounter{theorem}{1}}
\def\beq{\begin{eqnarray*}}
\def\eeq{\end{eqnarray*}}
\def\se{\setcounter{equation}{\value{theorem}}} 
\def\att{\addtocounter{theorem}{1}}
\def\vs{\vskip 5pt}

\def\bl{\begin{lemma}}
\def\el{\end{lemma}}
\def\bt{\begin{theorem}}
\def\et{\end{theorem}}
\def\bprop{\begin{prop}}
\def\eprop{\end{prop}}
\def\bd{\begin{definition}}
\def\ed{\end{definition}}
\def\br{\begin{remark}}
\def\er{\end{remark}}
\def\bexer{\begin{exercise}}
\def\eexer{\end{exercise}}

\newtheorem{theorem}{Theorem}[section]
\newtheorem{proposition}[theorem]{Proposition}
\newtheorem{lemma}[theorem]{Lemma}

\newtheorem{definition}[theorem]{Definition}

\newtheorem{question}[theorem]{Question}

\newtheorem{example}[theorem]{Example}

\renewcommand\Re{\mathrm{Re\, }}

\newcommand{\cK}{\mathcal K}
\newcommand{\bN}{\mathbb N}
\newcommand{\bB}{\mathbb B}
\newcommand{\bD}{\mathbb D}
\newcommand{\ol}[1]{\overline{#1}}

\DeclareMathOperator{\Mult}{Mult}

\newcommand\smf{\vertiii{f}}

\newcommand\smg{\vertiii{ g }}
\newcommand\smfc{\vertiii{\phic }}
\newcommand\npu{N^+_u}
\newcommand\nph{N^+(\htd)}
\newcommand\nevu{N_u}
\newcommand\oma{\omega_\alpha}
\newcommand\omat{\omega_\alpha^2}
\newcommand\mm{{\rm Mult}(H^2_d)}
\newcommand\phic{{\phi_c}}
\newcommand\tms{T_m^*}
\newcommand\htd{{H^2_d}}
\newcommand\ran{{\rm ran}}
\newcommand\Inl{{\I}_{\rm nlc}}

\def\IH{{\mathcal I}}
\def\II{{\mathcal I}_{\rm nlc}}

\newcommand\dF{d_F}
\newcommand\dL{d_L}

\newtheorem{lem}[theorem]{Lemma}
\newtheorem{prop}[theorem]{Proposition}
\newtheorem{cor}[theorem]{Corollary}

\theoremstyle{definition}

\theoremstyle{remark}
\newtheorem{rem}[theorem]{Remark}

\newtheorem*{rem*}{Remark}
\newtheorem*{exa*}{Example}

\title[Common Range]{The common range of co-analytic Toeplitz operators on the Drury-Arveson space}
\date{\today}
\author[A. Aleman]{Alexandru Aleman}
\address{Lund University, Mathematics, Faculty of Science, P.O. Box 118, S-221 00 Lund, Sweden}
\email{alexandru.aleman@math.lu.se}

\author[M. Hartz]{Michael Hartz}
\address{Fachrichtung Mathematik, Universit\"at des Saarlandes, 66123 Saarbr\"ucken, Germany}
\email{hartz@math.uni-sb.de}
\thanks{M.H. was partially supported by a GIF grant.}

\author[J. M\raise.5ex\hbox{c}Carthy]{John E. M\raise.5ex\hbox{c}Carthy}
\address{Department of Mathematics, Washington University in St. Louis, One Brookings Drive,
	St. Louis, MO 63130, USA}
\email{mccarthy@wustl.edu}
\thanks{J.M. was partially supported by National Science Foundation Grant DMS 2054199}

\author[S. Richter]{Stefan Richter}
\address{Department of Mathematics, University of Tennessee, 1403 Circle Drive, Knoxville, TN 37996-1320, USA}
\email{srichter@utk.edu}

\begin{document}

\bibliographystyle{plain}
\maketitle

\begin{abstract}
  We characterize the common range of the adjoints of cyclic multiplication operators on the Drury--Arveson space.
  We show that a function belongs to this common range if and only if its Taylor coefficients
  satisfy a simple decay condition.
  To achieve this, we introduce the uniform Smirnov class on the ball and determine its dual space.
We show that the dual space of the uniform Smirnov class equals the dual space of the strictly smaller
Smirnov class of the Drury-Arveson space, and that this in turn equals the common range of the adjoints of cyclic multiplication operators.
\end{abstract}

\section{Introduction}
\label{secintro}

\subsection{Known results on the disk}

Let $H^2$ denote the Hardy space on the unit disk $\D$, and let $H^\infty$ denote the bounded analytic
functions on $\D$. If $m \in H^\infty$ we let $T_m : H^2 \to H^2,\ f \mapsto mf $ denote the multiplication operator. We call its adjoint $T_m^*$ a co-analytic Toeplitz operator. In \cite{hel90}, H. Helson asked
\begin{question}
\label{quA1}
What functions are in the range of $T_m^*$ for every outer function $m$ in $H^\infty$?
\end{question}

He was led to this question from the following considerations.
 The Nevanlinna class of $\D$  is defined as
\be
\label{eqa1}
N(\D) \= \Big\{ f : f\ {\rm holomorphic \ on\ } \D, \
  \sup_{0 < r < 1} \int \log(1 + |f(re^{i\theta})|)d \theta < \infty
\Big\} .
\ee
The Smirnov class $N^+(\D)$ is a subset of functions in $N(\D)$:
\be
\label{eqa2}
N^+(\D)
 \= \Big\{ f \in N(\D) \ : \  \lim_{r \uparrow 1} \ \int
 \log [ 1 + |f( e^{i\theta} ) - f(r  e^{i\theta} )| ] d\theta  \= 0  \Big\},
\ee
where $f( e^{i\theta} )$ is the non-tangential limit, which exists a.e. \cite[Thm. II.5.3]{gar81}.
It can be shown (see \eg \cite[Sec. II.5]{gar81}) that
\se\att
\begin{eqnarray}
\label{eqa3}
N(\D) &\=& \{ f = g/m : \ g \in H^2,\  m \in H^\infty, m\neq 0 {\rm  \ on\ } \D \} \\
\att
N^+(\D) &\=& \{ f = g/m : \ g \in H^2, \  m\ {\rm outer\ } \}
\label{eqa4}\\
\nonumber
&=&   \{ f = q/m : \ q \in H^\infty,  m\ {\rm outer\ } \} .
\end{eqnarray}
We shall use {\em outer} throughout the note to mean an outer function in $H^\infty$.
Helson considered
the locally convex inductive limit topology $\I$ on $N^+(\D)$ that comes from viewing it
as
\[
N^+(\D) \= \bigcup \Big\{ \frac{1}{m} H^2 \ : \   m\ {\rm outer\ } \Big\} .
\]
First we put a norm on each $\frac{1}{m} H^2$ by declaring
\[
\| g \|_{\frac{1}{m} H^2} \ := \ \| mg \|_{H^2}  ,
\]
(so dividing by $m$ is a unitary from $ H^2$ onto $\frac{1}{m} H^2$).
Then $\I$ is the translation-invariant topology that has a neighborhood base at zero
consisting of all absolutely convex sets $\Omega$ with the property that
$\Omega \cap \frac{1}{m} H^2$ is open for each outer function $m$
(this is the same as saying that $m \Omega \cap H^2$ is open in $H^2$ for every outer
$m$).
In other words, $\mathcal{I}$ is the finest locally convex topology
on $N^+(\mathbb{D})$ that makes all inclusions $\frac{1}{m} H^2 \hookrightarrow N^+(\mathbb{D})$ continuous.
Helson proved in \cite{hel90} the following theorem:
\bt
Let $h(z) = \sum_{n=0}^\infty \gamma_n z^n$ be in $H^2$. Then the functional
\[
\Gamma: \sum b_n z^n \ \mapsto \ \sum b_n \overline{\gamma_n}
\]
is a continuous linear functional on $(N^+(\D), \I)$ if and only if
$h$ is in the range of $T_m^*$ for every  outer function $m$.
\et
He asked for an intrinsic characterization of such $h$. One can view the condition of Question \ref{quA1} in some ways as a smoothness condition.
Indeed, if $m(z) = 1-z$ and $\tms f = h,$ then denoting the $k$-th Taylor coefficient of a holomorphic
function $g$ by $\widehat{g}(k)$,
we have
$\widehat{h} (k) = \widehat{f}(k) - \widehat{f} (k+1)$. So $h$ has a telescoping Taylor series when evaluated at $1$, and
 by Abel's theorem $\lim_{r \uparrow 1} h(r 1)$ exists. Similarly looking at symbols that are powers
of $(e^{i\theta} - z)$, one can show that if $h$ is  in $\bigcap_{m\ {\rm outer}} {\rm ran}(T_m^*)$, all its derivatives have radial limits everywhere on the unit circle.

There is another topology on $N^+(\mathbb{D})$. One can define a metric on $N(\D)$ by
\[
\rho(f,g) \ = \
 \sup_{0< r < 1} \int \log ( 1 + |[f-g] (r e^{i \theta}) |) d \theta .
\]
The Nevanlinna class with the metric $\rho$ is not a topological vector space, because
scalar multiplication is not continuous \cite{sha-shi}. On the smaller Smirnov class,
$\rho$ makes it into a topological vector space with a complete translation invariant metric
(however it is not locally convex).  The dual of $(N^+(\D), \rho)$ had been found by N. Yanagihara \cite{yan73a}:
\bt
The functional
\[
\Gamma: \sum b_n z^n \ \mapsto \ \sum b_n \overline{\gamma_n}
\]
is a continuous linear functional on $(N^+(\D), \rho)$ if and only if
there exist constants $M,c > 0$ so that
\be
\label{eqA6}
| \gamma_n | \ \leq \ M e^{-c \sqrt{n}}  \quad \text{ for all } n \in \mathbb{N}.
\ee
\et
In \cite{mcc90c} it was shown that the inclusion map from $(N^+(\D), \rho)$ to $(N^+(\D), \I)$
is continuous, and moreover $\I$ is the finest locally convex topology such that inclusion is continuous.
Consequently, both spaces have the same dual, and hence Question \ref{quA1} was answered.
\bt
\label{thmcr1}
A function $h = \sum \hat h(n) z^n$ in  $ H^2$ is in the range of $T_m^*$ for every outer function $m$ in $H^\infty$ if and only if
there exist  $M,c > 0$ so that
\be
\label{eqA61}
| \hat h(n)  | \ \leq \ M e^{-c \sqrt{n}}  \quad \text{ for all } n \in \mathbb{N}.
\ee
\et

\begin{remark}
\label{remA1}
If $m$ is any non-zero function in $H^\infty$, the range of $T_m^*$ depends only on its outer part.
Indeed, factoring $m=u m_o$ where $u$ is inner and $m_o$ is outer, we have that
$T_m^* uf = T_{m_o}^* f$
and $T_m^*$ annihilates $(uH^2)^\perp$.
So Question \ref{quA1} is the same
as asking for a characterization of
\be
\label{eqA8}
\bigcap \{ {\rm ran}(T_m^*) \ : \ m \in H^\infty, \ {\rm not\ identically\ } 0 \}.
\ee
\end{remark}

\begin{remark}
Functions satisfying the smoothness condition \eqref{eqA6} come up in multiple places.
For example, in \cite{davmcc} they are shown to be the functions that are the multipliers of every de Branges-Rovnyak space $H(b)$ when $b$ is not an extreme point of the ball of $H^\infty$; see
\cite{sar94} for a treatment of these spaces. In \cite{bfv}, the authors give  bounds on the number
of zeroes such functions can have in $\overline{\D}$, and discuss how the class coincides with a certain class of Cauchy transforms.
\end{remark}

\subsection{New results on the ball}

Let $d$ be a positive integer, and $\B_d$ denote the unit ball in $\C^d$.
Let us define the Drury-Arveson space $H^2_d$ to be the Hilbert space of holomorphic functions
on $\B_d$ with reproducing kernel given by
\[
k(z, \lambda) \= \frac{1}{1 - \la z, \lambda \ra_{\C^d}} .
\]
We let ${\mathbb N}$ denote the natural numbers (including $0$).
The monomials $\{ z^\alpha : \alpha \in {\mathbb N}^d \}$ form an orthogonal basis for $H^2_d$. Their norms will come up
frequently, so we shall define
\be
\label{eqA13}
\oma \ :=\  \| z^\alpha \|_{\htd} \= \sqrt{\frac{\alpha_1 !\dots \alpha_d !}{(\alpha_1 + \dots + \alpha_d)!}} .
\ee
For many questions regarding multivariable operator theory, the Drury--Arveson space turns
out to be the right generalization of $H^2$ to the unit ball; see \cite{Arveson98, Shalit13} for background on $H^2_d$.

We shall let $\mm$ denote the multiplier algebra of $H^2_d$. For each $m$ in $\mm$, we shall
let \[
T_m : f \mapsto mf
\] denote the multiplication operator, and we shall call its adjoint $\tms$ a co-analytic Toeplitz operator on $\htd$. A multiplier $m$  is called cyclic if ${\rm ran}(T_m)$ is dense in $H^2_d$. Our goal is to answer the $d$-dimensional version of Question \ref{quA1}.
\begin{question}
\label{quAB1}
What functions are in the range of $\tms$ for every cyclic $m$ in $\mm$?
\end{question}
In analogy to \eqref{eqa3} and \eqref{eqa4}, we define the Nevanlinna and Smirnov classes
of $H^2_d$ by
\begin{eqnarray*}
\label{eqa7}
N(H^2_d) &=& \{ f = g/m : \ g \in H^2_d, m \in \mm, m\neq 0 {\rm  \ on\ } \B_d \} \\
N^+(H^2_d) &=& \{ f = g/m : \ g \in H^2_d, m \in \mm, m\ {\rm cyclic\ } \} .
\label{eqa8}
\end{eqnarray*}
In \cite{ahmrSm} we proved that $N^+(H^2_d)$ is equal to
\[
 \{ f = \phi/m : \ \phi \in \mm, m \in \mm, m\ {\rm cyclic\ } \} .
\]

As before, let us define
the locally convex inductive limit topology $\I$ on $N^+(H^2_d)$.
We norm each $\frac{1}{m} H^2_d$ by
\[
\| g \|_{\frac{1}{m} H^2_d} \ := \ \| mg \|_{H^2_d}.
\]
Then $\I$ is the translation-invariant topology that has a neighborhood base at zero
consisting of all absolutely convex sets $\Omega$ with the property that
$\Omega \cap \frac{1}{m} H^2_d$ is open for each cyclic multiplier $m$.
Equivalently, $\mathcal{I}$ is the finest locally convex topology on $N^+(H^2_d)$
that makes all inclusions $\frac{1}{m} H^2_d \hookrightarrow N^+(H^2_d)$ continuous.
The polynomials are dense in $(N^+(H^2_d),\mathcal{I})$, and
Helson's theorem remains true; we prove this in Section \ref{secd}.
\bt
\label{thmAB1}
The function $h \in H^2_d$ is in $\bigcap \{ {\rm ran}(\tms)  :  m \ {\rm cyclic} \}$ if and only if
the functional
\[
\Gamma_h : f \mapsto \la f , h \ra_{\htd}
\]
extends to be a continuous linear functional on $(\nph, \I)$.
\et
To characterize these functions by their Taylor coefficients, we need to study an appropriate version
of \eqref{eqa2}.
In \cite{rud80} W. Rudin defined the Nevanlinna and Smirnov classes on the unit ball
$\B_d$ in $\C^d$ analogously to \eqref{eqa1} and \eqref{eqa2}, replacing integration over the circle with
integration over the sphere. He conjectured that the two classes were equal for $d \geq 2$, but
the solution of the inner function problem on the ball showed that this was not true. See \cite{rud86}.

In this note we shall study holomorphic functions on the ball that are uniformly in the Smirnov class
on every disk through the origin. So we shall let
 $\nevu$ denote the holomorphic functions on the ball $\B_d$ such that
\be
\label{eqa41}
\smf \ : = \
\sup_{\zeta \in \partial {\mathbb B}_d} \sup_{0< r < 1} \int \log ( 1 + |f (r e^{i \theta} \zeta) |) \frac{d \theta}{2 \pi}
\ee
is finite.  This is not a norm, but we can define a metric by
\[
\rho(f,g) =  \vertiii{f-g} =
\sup_{\zeta \in \partial {\mathbb B}_d} \sup_{0< r < 1} \int \log ( 1 + |[f-g] (r e^{i \theta} \zeta) |) \frac{d \theta}{2 \pi} .
\]
We define the uniform Smirnov class $\npu$ by
\be
\label{eqa5}
\npu \= \{ f \in \nevu : \lim_{r \uparrow 1} \rho(f_r, f) = 0 \} ,
\ee
where $f_r(z) := f(rz)$.
We shall show in Proposition~\ref{lem:F-space} that $(\npu, \rho)$ is an F-space, \ie a topological vector space in which
the topology comes from a complete translation invariant metric, and that the polynomials are dense in $(\npu, \rho)$. Moreover we show that multiplication is continuous, so $(\npu, \rho)$ is a topological algebra.
(It can be shown that when $d=1$ the  definitions of the Smirnov class from \eqref{eqa2} and \eqref{eqa5}
coincide. We were unable to find a convenient reference, so we include a proof in Proposition \ref{prl1}.)

In Theorem \ref{thmH1} we prove that $(\nph, \I) $  embeds in
$(\npu, \rho)$, but the inclusion is not continuous, and  is proper whenever $d \geq 2$.
Indeed, in Section~\ref{secf} we prove that $N(H^2_d)$ does not even contain the ball algebra.
\bt
\label{cor:ball_nevanlinna}
 For all $d \geq 2$ there is a holomorphic function $f$
on $\B_d$ that is continuous up to the boundary, but is not in $N(H^2_d)$.
\et
Next, we find the dual of $(\npu, \rho)$. We prove in Section \ref{secc}:
\bt
\label{thmAB3}
The functional
\be
\label{eqA14}
\Gamma : f \mapsto \sum_{\alpha \in {\mathbb N}^d} \widehat{f}(\alpha) \overline{\gamma_\alpha}
\ee
is in $(\npu, \rho)^*$ if and only if there exist constants $M,c > 0$ so that
\be
\label{eqAB4}
| \gamma_\alpha| \ \leq \ M \oma e^{-c \sqrt{|\alpha|}} \quad \text{ for all } \alpha \in \mathbb{N}^d.
\ee
In this case, the series in \eqref{eqA14} converges absolutely for every $f \in N_u^+$.
\et
Since $\nph \subsetneq \npu$, and the inclusion is not even continuous,
 Theorem \ref{thmAB3} would not seem strong enough to answer
Question \ref{quAB1}. Nevertheless, our most surprising result is that it is, and $(\nph,\I)$ and
$(\npu, \rho)$ have the same duals.
Putting all these results together we get our principal result.
\bt
\label{thmH2}
Let $h \in H^2_d$ and let $\Gamma = \langle \cdot,h \rangle_{\htd}$ be the associated
  linear functional. The following assertions are equivalent:
  \begin{enumerate}[label=\normalfont{(\roman*)}]
    \item $\Gamma \in (N^+_u,\rho)^*$,
    \item $\Gamma \in (N^+(H^2_d),\mathcal{I})^*$,
    \item $h \in \ran(T_m^*)$ for every cyclic $m \in \Mult(H^2_d)$,
    \item there exist $M,c > 0$ so that
      \be
      \label{eqAB5}
        |\widehat{h}(\alpha)| \oma \le M e^{-c \sqrt{|\alpha|}}
      \ee
      for all $\alpha \in \mathbb{N}^d$.
  \end{enumerate}
  In this case, the functional $\Gamma$ on $N_u^+$ or on $N^+(H^2_d)$ is given
  by
  \begin{equation*}
    \Gamma(f)
=\sum_{n = 0}^\infty \la f_n,h_n\ra_{H^2_d}= \sum_{n=0}^\infty\sum_{|\alpha|=n }\omega_\alpha^2 \widehat{f}(\alpha)\overline{\widehat{h}(\alpha)},
  \end{equation*}
  where both sums converge absolutely, and $f = \sum_{n=0}^\infty f_n$
  and $h = \sum_{n=0}^\infty h_n$ are the homogeneous expansions of $f$ and $h$, respectively.

\et
Note: The reason $\omega_\alpha$ appears on the left in \eqref{eqAB5} and on the right in \eqref{eqAB4}
is the way we define $\Gamma$. In Theorem \ref{thmH2} we use the $\htd$ inner product, which
is the form we will use in our proofs, and
in Theorem \ref{thmAB3} we used the $\ell^2$ inner product on the coefficients, just to make the statement of the theorem more succinct.
This means that $\gamma_\alpha = \omat \widehat{h}(\alpha)$.

We prove Theorem~\ref{thmH2} in Sections \ref{secd} and \ref{secnec}.

In Section \ref{secfr}, we prove that $(\npu, \I)$ can be realized as a dense subspace
of a Fr\'echet space, i.e. a complete locally convex metrizable space, so in particular
the topology $\I$ on $\npu$ is metrizable.
In Section \ref{secil} we have an appendix on inductive limit topologies.

\section{Basic properties of \texorpdfstring{$\npu$}{the uniform Smirnov class}}
\label{secBA}

The following lemma is proved in \cite[II.3.1 and II.11.2]{prig}.
\bl
\label{lemb3}
Let $g \in N^+(\D)$. Then for any $z \in \D$, we have
\be
\label{eqaa5}
|g(z) | \ \leq \ e^{\frac{2 \smg}{1-|z|}} - 1.
\ee

Moreover,  its Taylor coefficients satisfy
\be
\label{eqaa4}
| \widehat g (n) | \ \leq \ e^{\sqrt{8 n \smg  }(1 + \vare_n )} ,
\ee
where $\vare_n$ is $o(1)$, depends only on $\smg$, and decreases as $\smg$ decreases.
\el

For $\zeta \in \partial \mathbb{B}_d$, let
\begin{equation*}
  i_\zeta: \mathbb{D} \to \mathbb{B}_d, \quad z \mapsto z \zeta.
\end{equation*}
If $f \in \mathcal{O}(\mathbb{B}_d)$, then $f \circ i_\zeta \in \mathcal{O}(\mathbb{D})$
is the slice function along the direction $\zeta$.
Clearly, we have $\vertiii{f \circ i_\zeta} \le \vertiii{f}$.
So if $f \in \npu$, then $f \circ i_\zeta \in N(\D)$ and $\lim_{r \to 1} \vertiii{f \circ i_\zeta - (f \circ i_{\zeta})_r} = 0$,
which implies that $f \circ i_{\zeta} \in N^+(\D)$ by Proposition \ref{prl1}.
So
we get from \eqref{eqaa5} that
\be
\label{eqaa6}
f \in \npu \ \Rightarrow \
|f(w) | \ \leq \ e^{\frac{2 \smf}{1-|w|}} - 1, \qquad \forall \ w \in \B_d .
\ee

\begin{prop}
  \label{lem:F-space}
  $(\npu, \rho)$ is an F-space, and multiplication is continuous.
  Moreover,  the polynomials are dense.
\end{prop}

\begin{proof}
  It is straightforward to check that $\rho$ defines a metric on $N_u^+$.
  Observe that if $f \in N_u$,
  $\lambda \in \mathbb{C}$ and $n \in \mathbb{N}$ with $|\lambda| \le n < |\lambda| + 1$, then
  \begin{equation*}
    \vertiii{\lambda f} \le \vertiii{n f} \le n \vertiii{f} \le (|\lambda|+1) \vertiii{f}
  \end{equation*}
  by the triangle inequality.
  From this, it easily follows that $\npu$ is closed under scalar multiplication, and hence is a  subspace of $N_u$.

 To see density of the polynomials, note that if $f \in \npu$, then for each $r < 1$, the function
  $f_r$ is holomorphic in a neighborhood of $\overline{\mathbb{B}_d}$ and hence a uniform limit of polynomials on $\overline{\mathbb{B}_d}$.
  Thus each $f_r$ and therefore $f$ belongs to the $\rho$-closure of the polynomials.

  Continuity of addition on $\npu$ follows from the triangle inequality.
To see continuity of  multiplication, from which scalar multiplication follows as a special case,
note that
if $f,g,f_0,g_0\in \npu$ then
 $$fg-f_0g_0=(f-f_0)(g-g_0)+f_0(g-g_0)+g_0(f-f_0).$$
Using  the triangle inequality together with the inequality  $\vertiii{FG}\le
\vertiii{F}+\vertiii{G}$,
we obtain
$$\vertiii{fg-f_0g_0}\le \vertiii{f-f_0}+\vertiii{g-g_0} +\vertiii{f_0(g-g_0)}+\vertiii{g_0(f-f_0)}.$$
So it suffices to prove that multiplication is separately continuous.

Since   $\vertiii{f g} \le ( \|f\|_{\infty}+1) \vertiii{g}$, we have
\beq
\vertiii{f (g-g_0)} & \ \le \ &  \vertiii{(f - f_r) (g - g_0)} + \vertiii{f_r (g - g_0)} \\
                    &     \le & \vertiii{f - f_r} + \vertiii{g - g_0} + (\|f_r\|_\infty +1) \vertiii{g - g_0}).
 \eeq
Let $\vare > 0$, and choose $r < 1$ so that $\vertiii{f-f_r} < \frac{\vare}{3}$.
Then if
\[
\vertiii{g - g_0} < \frac{\vare}{3+ 3 \| f_r \|_\infty}
\]
we get $ \vertiii{f(g - g_0)} < \vare $. So multiplication is continuous.

  Finally, we show completeness. Let $(f_n)$ be a Cauchy sequence in $\npu$.
 From \eqref{eqaa6}  we see
  that $(f_n)$ converges uniformly
  on compact subsets of $\mathbb{B}_d$ to a holomorphic function $f$.
  On the other hand, completeness of $N^+(\D)$ (see \cite[Theorem 1]{yan73a} or \cite[p. 919]{sha-shi})  yields that each slice sequence $(f_n \circ i_\zeta)$ converges in $N^+(\D)$
  and in particular pointwise on $\mathbb{D}$, so that the limit necessarily equals $f \circ i_\zeta$.
  Moreover, as $  \vertiii{ f_n \circ i_\zeta - f_m \circ i_\zeta} \le \vertiii{f_n - f_m}$
   it follows that the convergence is uniform in $\zeta \in \partial \mathbb{B}_d$,
  hence $f \in N_u$ and $\lim_{n \to \infty} \rho(f_n,f) = 0$.

To see that $f \in \npu$, let $\varepsilon > 0$ and let $n \in \mathbb{N}$
  with $\vertiii{f - f_n} < \varepsilon$. Then
  \begin{equation*}
    \vertiii{f - f_r} \le \vertiii{f - f_n} + \vertiii{f_n - (f_n)_r} + \vertiii{ (f_n)_r - f_r}
    \le 2 \varepsilon + \vertiii{f_n - (f_n)_r},
  \end{equation*}
  so $\limsup_{r \to 1} \vertiii{f - f_r} \le 2 \varepsilon$. Since $\varepsilon > 0$ was arbitrary,
  $f \in \npu$.
\end{proof}

We shall need the next result in Section \ref{secnec}.
\begin{lem}
  \label{lem:norm_equiv}
  Let $h(z) = \sum_{|\alpha| = n} \widehat{h}(\alpha) z^\alpha$ be a homogeneous polynomial of degree $n$. Then the following inequalities hold.
  \begin{enumerate}[label=\normalfont{(\alph*)}]
    \item $\|h\|_\infty \le \|h\|_{H^2_d} \lesssim (n+1)^{(d-1)/2} \|h\|_\infty$.
    \item $\max_{\alpha} |\widehat{h}(\alpha)| \omega_\alpha \le \|h\|_{H^2_d} \lesssim (n+1)^{(d-1)/2}
      \max_{\alpha} |\widehat{h}(\alpha)| \omega_\alpha$.
  \end{enumerate}
  Here, the implied constants only depend on $d$.
\end{lem}

\begin{proof}
  (a) We have
  \begin{equation*}
    \|h\|_\infty \le \|h\|_{\Mult(H^2_d)} = \|h\|_{H^2_d},
  \end{equation*}
  since $h$ is homogeneous (see e.g.\ \cite[Proposition 6.4]{Hartz17a}). Next,
  if $H^2(\partial \mathbb{B}_d)$, denotes the Hardy space on the ball, then the formula for the norm
  of a monomial in $H^2(\partial \mathbb{B}_d)$ in \cite[Proposition 1.4.9]{rud80} shows that
  \begin{equation*}
    \|h\|_{H^2_d}^2 = \binom{d-1+n}{d-1} \|h\|_{H^2(\partial \mathbb{B}_d)}^2
    \le
    \binom{d-1+n}{d-1} \|h\|_{\infty}^2
    \lesssim (n+1)^{d-1} \|h\|_{\infty}^2.
  \end{equation*}

  (b) Note that
  \begin{equation*}
    \|h\|^2_{H^2_d} = \sum_{|\alpha| = n} |\widehat{h}(\alpha)|^2 \omega_\alpha^2,
  \end{equation*}
  so the inequality on the left is obvious, and the inequality on the right follows
  from the fact that there are $\binom{d-1+n}{d-1} \lesssim (n+1)^{d-1}$ monomials in $d$ variables of degree $n$.
\end{proof}

In particular, the final condition in Theorem \ref{thmH2} admits the following equivalent reformulations.
\begin{cor}
  \label{cor:decay_equiv}
  Let $h = \sum_{\alpha} \widehat{h}(\alpha) z^\alpha \in \mathcal{O}(\mathbb{B}_d)$ and let $h_n = \sum_{|\alpha| = n} \widehat{h}(\alpha) z^\alpha$
  be the homogeneous component of degree $n$. Then the following assertions are equivalent:
  \begin{enumerate}[label=\normalfont{(\roman*)}]
    \item There exist $M,c > 0$ so that
      \begin{equation*}
        |\widehat{h}(\alpha)| \omega_\alpha \le M e^{-c \sqrt{|\alpha|}} \quad \text{ for all } \alpha \in \mathbb{N}^d;
      \end{equation*}
    \item there exist $M,c > 0$ so that
      \begin{equation*}
        \|h_n\|_\infty \le M e^{-c \sqrt{n}} \quad \text{ for all } \alpha \in \mathbb{N}^d;
      \end{equation*}
    \item there exist $M,c > 0$ so that
      \begin{equation*}
        \|h_n\|_{H^2_d} \le M e^{-c \sqrt{n}} \quad \text{ for all } \alpha \in \mathbb{N}^d.
      \end{equation*}
  \end{enumerate}
\end{cor}

\begin{proof}
  Clear from Lemma \ref{lem:norm_equiv}.
\end{proof}

Yanagihara \cite{yan74} showed that the Taylor coefficients of a function $f \in N^+(\mathbb{D})$ satisfy
the growth estimate $|\widehat{f}(n)| = O(\exp(c \sqrt{n}))$ for all $c > 0$.
We will use the following generalization to $N_u^+$ of this fact in Section \ref{secc}.
\begin{lem}
  \label{lem:Smirnov_growth}
  Let $f \in \npu$ have homogeneous decomposition $f = \sum_{n=0}^\infty f_n$.
  Then for all $c > 0$, there exists $M \ge 0$ so that $\|f_n\|_{H^2_d} \le M e^{c \sqrt{n}}$.
\end{lem}

\begin{proof}
  By \eqref{eqaa4}, for each $c > 0$  there exists $\delta > 0$ so that if  $\vertiii{g} < \delta$ then
    \begin{equation*}
    | \widehat{g}(n)| \le e^{c \sqrt{n}} .
  \end{equation*}

  Since $\lim_{\varepsilon \to 0} \vertiii{\varepsilon f} = 0$,
  there exists $\varepsilon > 0$ so that $\vertiii{\varepsilon f} < \delta$, hence if
  $f_\zeta(z)  = f(z \zeta)$ denotes the slice function, then $\vertiii{\varepsilon f_\zeta} < \delta$
  for all $\zeta \in \partial \mathbb{B}_d$. Consequently,
  \begin{equation*}
    \varepsilon |f_n(\zeta)| = | \widehat{ \varepsilon f_\zeta}(n)| \le e^{c \sqrt{n}}.
  \end{equation*}
  This shows that $\|f_n\|_\infty = O(e^{c \sqrt{n}})$ for all $c > 0$. The claim now follows from Corollary \ref{cor:decay_equiv}.
\end{proof}

\section{\texorpdfstring{$N^+(H^2_d)$ is contained  in $\npu$}{Containment of Smirnov classes}}
\label{secB}
Our goal is to show that $N^+(H^2_d) \subset \npu$. We begin with two lemmata.

\begin{lem}
  \label{lem:DA_Smirnov_inclusion}
  $H^2_d \subset \npu$ and
  $\vertiii{f} \le \|f\|_{H^2_d}$ for all $f \in H^2_d$.
\end{lem}

\begin{proof}
  The inequality $\log(1+x) \le x$ shows that for all $\zeta \in \partial \mathbb{B}_d$, we have
  \begin{equation*}
    \sup_{0 < r < 1} \int_{0}^{2 \pi} \log(1 + |f(r e^{i \theta} \zeta)|) \frac{d \theta}{2 \pi}
    \le \|f \circ i_\zeta\|_{H^1} \le \|f \circ i_\zeta\|_{H^2} \le \|f\|_{H^2_d}.
  \end{equation*}
  Hence $H^2_d \subset N_u$, and the inequality in the statement holds.
  Moreover, $\rho(f_r,f) \le \|f - f_r\|_{H^2_d} \to 0$ as $r \to 1$, so $f \in \npu$.
\end{proof}

\begin{lem}
  \label{lem:vNI}
  Let $\mathcal{H}$ be a reproducing kernel Hilbert space and let $\psi \in \Mult(\mathcal{H})$ with $\|\psi\|_{\Mult(\mathcal{H})} \le 1$.
  Then for all $r \in (0,1)$, we have
  \begin{equation*}
    \Big \| \frac{1 - \psi}{ 1- r \psi} \Big\|_{\Mult(\mathcal{H})} \le \frac{2}{1 + r}.
  \end{equation*}
\end{lem}

\begin{proof}
  By von Neumann's inequality,
  \begin{equation*}
    \Big \| \frac{1 - \psi}{ 1- r \psi} \Big\|_{\Mult(\mathcal{H})} \le \sup_{|z| = 1} \Big| \frac{1 - z}{1 -r z} \Big|.
  \end{equation*}
  Straightforward calculus shows that the supremum on the right is attained for $z=-1$ and hence
  equal to $\frac{2}{1 + r}$.
\end{proof}

A normalized complete Pick space can be defined to be a Hilbert function space on a set $X$ with a reproducing kernel
of the form
\[
k(x,y) \= \frac{1}{1- \la b(x), b(y) \ra},
\]
where $b$ is some function from $X$ into the unit ball of a Hilbert space such that $b(x_0) = 0$
for some base point $x_0 \in X$.
See \cite{agmc_cnp,ampi} for how this is equivalent to having the complete Pick property.
The multiplier algebra $\Mult(\mathcal{H})$ can be identified with a WOT-closed subalgebra of $B(\mathcal{H})$.
In this way, $\Mult(\mathcal{H})$ becomes a dual space. On bounded subsets of $\Mult(\mathcal{H})$, the resulting
weak-$*$ topology agrees with the topology of pointwise convergence on $X$.
Indeed, this follows from density of the linear span of all kernel functions in $\mathcal{H}$.

It follows from a result of Davidson, Ramsay and Shalit, see \cite[Corollary 2.7]{drs15}, that there is a 1-to-1 correspondence
between multiplier invariant subspaces of a complete Pick space
and weak-$*$ closed ideals of the multiplier algebra.
In particular, if $m$ is a cyclic multiplier, then $m \Mult(\mathcal{H})$ is weak-$*$ dense
in $\Mult(\mathcal{H})$. We require the following Kaplansky density type refinement of this fact.
If $f \in \mathcal{H}$, we shall let $[f]$ denote the closure in $\mathcal{H}$ of $\{ mf : m \in  \Mult(\mathcal{H})\}$.

\begin{lem}
  \label{lem:cyclic_Kaplansky}
  Let $\mathcal{H}$ be a normalized complete Pick space and let $m \in \Mult(\mathcal{H})$.
  If $\varphi \in \Mult(\mathcal{H}) \cap [m]$ with $\|\varphi\|_{\Mult(\mathcal{H})} \le 1$,
  then there exists a sequence $(\varphi_n)$ in $\Mult(\mathcal{H})$ so that
  \begin{enumerate}
    \item $(\varphi_n m)$ converges to $\varphi$ in the weak-$*$ topology, and
    \item $\|\varphi_n m\|_{\Mult(\mathcal{H})} \le 1$ for all $n \in \mathbb{N}$.
  \end{enumerate}
\end{lem}

\begin{proof}
  Since $\varphi \in [m]$, there exists a sequence $(q_n)$ in $\Mult(\mathcal{H})$
  such that
  \be
\label{eqB1}
    \sum_{n=1}^\infty \| n^2 (q_n m -\varphi)\|^2_{\mathcal{H}} \le 1.
\ee
  In this setting, \cite[Theorem 1.1]{ahmrfac} yields a sequence $(\psi_n)$ in $\Mult(\mathcal{H})$
  forming a contractive column multiplier and a multiplier
  $\psi \in \Mult(\mathcal{H})$ with $\|\psi\|_{\Mult(\mathcal{H})} \le 1$ and $\psi \neq 1$ so that
  \begin{equation*}
    n^2 (q_n m - \varphi) = \frac{\psi_n}{1 - \psi} \quad \text{ for all } n \ge 1.
  \end{equation*}
  In particular, $\|\psi_n\|_{\Mult(\mathcal{H})} \le 1$ and thus
  \begin{equation}
    \label{eqn:mult_norm_small}
    \| (1 - \psi) (q_n m - \varphi)\|_{\Mult(\mathcal{H})} \le \frac{1}{n^2}.
  \end{equation}
  Let $r_n = 1 - \frac{1}{n}$ and set
  \begin{equation*}
    \varphi_n = \frac{1 - \psi}{1 - r_n \psi} q_n \in \Mult(\mathcal{H}).
  \end{equation*}
  Then
  \begin{align*}
    \|\varphi_n m\|_{\Mult(\mathcal{H})}
    &= \Big\| \frac{1 - \psi}{1 - r_n \psi} q_n m \Big\|_{\Mult(\mathcal{H})} \\
    &\le
    \Big\| \frac{1 - \psi}{1 - r_n \psi} (q_n m - \varphi) \Big\|_{\Mult(\mathcal{H})}
    + \Big\| \frac{(1 - \psi)\varphi}{1 - r_n \psi} \Big\|_{\Mult(\mathcal{H})}.
  \end{align*}
  We estimate the first summand using \eqref{eqn:mult_norm_small} and the second summand using
  Lemma \ref{lem:vNI} to find that
  \begin{equation*}
    \|\varphi_n m\|_{\Mult(\mathcal{H})}
    \le \frac{1}{n^2 (1 - r_n)} + \frac{2}{1 + r_n}
    = \frac{1}{n} + \frac{2}{2 - \frac{1}{n}},
  \end{equation*}
  hence $\limsup_{n \to \infty} \|\varphi_n m\|_{\Mult(\mathcal{H})} \le 1$.

  Moreover, since $(q_n m)$ tends to $\varphi$ pointwise, we see that $(\varphi_n m)$ tends to $\varphi$ pointwise
  and hence in the weak-$*$ topology. Replacing $\varphi_n$ with $t_n \varphi_n$ for a suitable
  sequence of scalars $(t_n)$ in $(0,1)$ converging to $1$, we obtain the desired sequence.
\end{proof}
\begin{remark} If $J$ is any weak-$*$ closed ideal in $\Mult(\mathcal{H})$ and $\varphi \in \Mult(\mathcal{H})$ is in the closure of $J$ in $\mathcal{H}$, then a very similar argument shows
that $\varphi$ is in $J$ (just replace $q_nm$ in \eqref{eqB1} by $m_n$ for some sequence $m_n$ in $J$). This yields another proof of one direction of  the Davidson-Ramsey-Shalit result.
For the other direction, we need to show
 that if we start with an invariant subspace $\M$, intersect
it with $\Mult(\mathcal{H})$ to get an ideal $J$, and then take the closure of $J$ in $\mathcal{H}$, we get
$\M$ back. This can also be proved using the representation $f = \frac{\varphi}{1 -\psi}$ for functions in $\mathcal{H}$;
see \cite[Cor. 4.1]{ahmrfac}.
\end{remark}

We are now able to prove that $N^+(H^2_d)$ is contained in $N_u^+$.

\begin{theorem}
  \label{thmH1}
  Let $m \in \Mult(H^2_d)$ by cyclic. Then $\frac{1}{m} H^2_d \subset \npu$, and the inclusion
  is continuous. Hence $N^+(H^2_d) \subset \npu$.
\end{theorem}

\begin{proof}
  Since $m$ is cyclic, $H^2_d$ is a dense subspace of $\frac{1}{m} H^2_d$, and we know from Lemma \ref{lem:DA_Smirnov_inclusion} that $H^2_d \subset \npu$.
  Moreover, $(\npu,\rho)$ is an F-space  by Proposition \ref{lem:F-space},
  so it suffices to show that the inclusion $H^2_d \subset \npu$ is continuous at $0$ with respect to the norm $\|\cdot\|_{\frac{1}{m} H^2_d}$ on $H^2_d$.
  Indeed, assuming this has been shown, then the inclusion $H^2_d \subset \npu$ is uniformly
  continuous with respect to $\|\cdot\|_{\frac{1}{m} H^2_d}$ and $\vertiii{\cdot}$,
  hence by completness of $\npu$, it extends to a continuous inclusion
from $ \frac{1}{m} H^2_d$ to $N^+_u$.

  Thus, we have to show the following statement:
  For each $\varepsilon > 0$, there exists $\delta > 0$ so that $f \in H^2_d$ and
  $\|m f\|_{H^2_d} < \delta$ implies $\vertiii{f} < \varepsilon$.
  To see this, let $\zeta \in \partial \mathbb{B}_d$. Since the slice function $f \circ i_\zeta$
  is in the Smirnov class on $\mathbb{D}$, we see that
  \begin{equation*}
    \sup_{0 < r < 1} \int_{0}^{2 \pi} \log( 1 + |f( r \zeta e^{i \theta})|) \frac{d \theta}{2 \pi}
    = \int_{0}^{2 \pi} \log( 1 + |f(\zeta e^{ i \theta})|) \frac{d \theta}{2 \pi}.
  \end{equation*}
  If $\varphi \in \Mult(H^2_d)$  with $\varphi(0) \neq 0$ and $\|\varphi m\|_{\Mult(H^2_d)} \le 1$,
  then $(m \varphi) \circ i_\zeta$ is non-zero almost everywhere on $\partial \mathbb{D}$.
  Using the  inequality $\log(1 + a b) \le \log(1 + a) + \log(b)$ for $a \ge 0$ and $b \ge 1$,
  we therefore find that
  \begin{align*}
    \log(1 + |f(\zeta e^{i \theta})|)
    &=
  \log \Big(1 + \Big| \frac{f(\zeta e^{i \theta}) m(\zeta e^{i \theta}) \varphi(\zeta e^{i \theta})}{m(\zeta e^{i \theta}) \varphi(\zeta e^{i \theta})} \Big| \Big) \\
    &\le \log (1 + | {f(\zeta e^{i \theta}) m(\zeta e^{i \theta}) \varphi(\zeta e^{i \theta})}) | )
  - \log |m(\zeta e^{i \theta}) \varphi(\zeta e^{i \theta})|
  \end{align*}
  for almost every $\theta$.
  Integrating in $\theta$, using Lemma \ref{lem:DA_Smirnov_inclusion} on the first summand and Jensen's inequality for holomorphic
  functions on the second summand, it follows that
  \begin{align*}
    \sup_{0 < r < 1} \int_{0}^{2 \pi} \log( 1 + |f( r \zeta e^{i \theta})|) \frac{d \theta}{2 \pi}
    \le \|f m \varphi\|_{H^2_d} - \log |m(0) \varphi(0)|.
  \end{align*}
  Taking the supremum over all $\zeta \in \partial \mathbb{B}_d$, we obtain
  \begin{equation}
    \label{eqn:cont}
    \vertiii{f} \le \|f m\|_{H^2_d} \|\varphi\|_{\Mult(H^2_d)} - \log |m(0) \varphi(0)|,
  \end{equation}
  which is true for any $\varphi \in \Mult(H^2_d)$ with $\varphi(0) \neq 0$
  and $\|\varphi m\|_{\Mult(H^2_d)} \le 1$.
  Now, let $\varepsilon > 0$. Since $m$ is cyclic, Lemma \ref{lem:cyclic_Kaplansky} yields
  a multiplier $\varphi \in \Mult(H^2_d)$ with $\|\varphi m\|_{\Mult(H^2_d)} \le 1$
  and $\log |m(0) \varphi(0)| > - \varepsilon/2$. Set $\delta = \varepsilon/ (2\|\varphi\|_{\Mult(H^2_d)})$.
  Then \eqref{eqn:cont} shows that if $\|f m \|_{H^2_d} < \delta$, then $\vertiii{f} <  \varepsilon$.
\end{proof}

An alternate way to prove Theorem \ref{thmH1} is to let $X = \htd$ in the following theorem.
We define the norm on $\frac1{g}X$ by $\|\frac{f}{g}\|_{\frac1{g}X}:= \|f\|_X$.

\begin{theorem}\label{inclusion} Suppose that $X$ is a Banach space of holomorphic functions
on $\B_d$ that contains the constants and such that ${\rm Mult}(X)$ is dense in $X$.
If $X$ is
continuously contained in $\npu$, and  $g\in X$ is cyclic,  then $\frac1{g}X$ is continuously contained in $\npu$.
\end{theorem}

\begin{proof} We prove first that $\frac1{g}X$ is continuously contained in $N_u$. To this end,
let $\varphi\in \Mult(X)$.  If $\z\in\partial\B_d$,  let $v_\z$ be the outer function
on $\D$  with $|v_\z(e^{i\theta})|=\max\{1,|g\varphi(e^{i\theta}\z)|\}$ a.e.. Then  $|v_\z|\ge \max\{1,|g\varphi\circ i_\z|\}$ in $\D$, which gives
\be
\label{est1}
\vertiii{\frac{f}{g}}=\vertiii{\frac{f\varphi}{g\varphi}} \le \sup_{\z\in \B_d}\sup_{0<r<1}\int_0^{2\pi}\log\left(1+\left|\frac{f\varphi v_\z}{g\varphi}(re^{i\theta}\z)\right|\right)\frac{d\theta}{2\pi}.\ee
Using $(1+ab)\le(1+a)b,~a\ge 0,b\ge 1$,   with $b=\left|\frac{ v_\z}{g\varphi}(re^{i\theta}\z)\right|$, it follows that
\begin{align}\label{est2}&\int_0^{2\pi}\log\left(1+\left|\frac{f\varphi v_\z}{g\varphi}(re^{i\theta}\z)\right|\right)\frac{d\theta}{2\pi}
\le \int_0^{2\pi}\log(1+|f\varphi(re^{i\theta}\z)|)\frac{d\theta}{2\pi}\\&\nonumber
+\int_0^{2\pi}\log\left|\frac{ v_\z}{g\varphi}(re^{i\theta}\z)\right|\frac{d\theta}{2\pi}.
\end{align}
\att
Using the definition of $v_\z$ together with  the inequality $|a|\le 1+|a-1|$, we obtain
\beq
\int_0^{2\pi}\log|v_\z(re^{i\theta})|\frac{d\theta}{2\pi}&\ \le \ & \int_0^{2\pi}\log|v_\z(e^{i\theta})|\frac{d\theta}{2\pi}
\\ &\le &
\int_0^{2\pi}\log(1+|g\varphi\circ i_\z(e^{i\theta})-1|)\frac{d\theta}{2\pi}\\&\le&\vertiii{g\varphi-1}.
\eeq
Moreover, since $g\varphi\circ i_\z$ is analytic in $\D$,
$$\int_0^{2\pi}\log\left|\frac{1}{g\varphi}(re^{i\theta}\z)\right|\frac{d\theta}{2\pi}\le -\log|g\varphi(0)|.$$
Now let $\varepsilon>0$. Since  $g$ is cyclic and $X$ is continuously contained in $\npu$, we can choose $\varphi\in \Mult(X)$ such that
$$ \vertiii{g\varphi-1}-\log|g\varphi(0)|<\frac{\varepsilon}{2},$$
and from \eqref{est1} and \eqref{est2} we obtain for this choice of $\varphi$
\be
\label{est3} \vertiii{\frac{f}{g}}<\frac{\varepsilon}{2}+\vertiii{f\varphi}.
\ee
Since $\varphi$ is a multiplier and  $X\subset \npu$, there is $\delta>0$
such that $\|\frac{f}{g}\|_{\frac1{g}X}= \|f\|_X<\delta$ implies $\vertiii{f\varphi}<\frac{\varepsilon}{2}$.  \\
 Thus $\frac1{g}X$ is continuously contained in $N_u$ and by assumption, the inclusion maps the dense set $X$ into $\npu$. Since  $\npu$  is closed in $N_u$ by Proposition \ref{lem:F-space}, the result follows.
\end{proof}

\begin{remark}
Note that Theorem \ref{thmH1} shows that the set $N^+(\htd)$ is contained in $\npu$, but it does {\em not} show that the inclusion from $(N^+(\htd), \I)$ to $(\npu, \rho)$ is continuous.
Indeed, let $\Inl$ be the non locally convex  inductive limit topology on $\cup \frac{1}{m} \htd$,
where a neighborhood base at $0$ is give by {\em all} sets $\Omega$ such that $\Omega \cap \frac{1}{m} \htd$ is open for all cyclic multipliers $m$. Then Theorem \ref{thmH1} says that the
inclusion from  $(N^+(\htd), \Inl)$ to $(\npu, \rho)$ is continuous, since the inverse image under inclusion of every open set in $(\npu , \rho)$ is open in $\Inl$.

Even when $d =1$ the inclusion from $(N^+(\htd), \I)$ to $(\npu, \rho)$ is not continuous,
since the two sets coincide, the inclusion from $(N_u^+,\rho)$ to $(N^+(H^2_d),\mathcal{I})$ is continuous
by \cite[Theorem 2.1]{mcc90c}, but the topology induced by $\rho$
is not locally convex, as remarked  by Yanagihara in
 \cite{yan73a}.
It follows that the inclusion cannot be continuous for $d > 1$ either,
since otherwise restricting to functions that depend only on one coordinate would yield a contradiction
(by using Lemma \ref{lem:IL_inclusion} below).
\end{remark}


\section{The dual of \texorpdfstring{$(\npu, \rho)$}{the uniform Smirnov class}}
\label{secc}

In this section, we prove Theorem~\ref{thmAB3}. We shall assume throughout that $d$ is a fixed integer greater than $0$.

For $c > 0$, let
\be
\label{eqb7}
\phic(z) \ = \ \exp\left(\frac{c}{8}\frac{1+z}{1-z}\right) - 1 .
\ee
This function is {\em not} in $N^+(\D)$, nonetheless the following two facts hold
(the first is proved in \cite[II.11.2]{prig}, the second in \cite[Lemma 1.5]{mcc92a}).
\bl
\label{lemb5}
With $\phic$ defined by \eqref{eqb7}, we have for $n>0$
\[
 \widehat \phic (n) \= e^{c \sqrt{n}(1 + o(1))} .
 \]
 Moreover,
 \[
 \lim_{c \downarrow 0} \smfc \= 0 .
 \]
 \el

For  $\zeta \in \mathbb{B}_d$, let
\begin{equation*}
  P_\zeta: \mathbb{C}^d \to \mathbb{C}, \quad z \mapsto \langle z,\zeta \rangle,
\end{equation*}
denote the orthogonal projection onto $\mathbb{C} \zeta$.
If $\Gamma$ is a continuous linear functional on $H^2_d$, let
\begin{equation*}
  \Gamma_\zeta : H^2 \to \mathbb{C}, \quad f \mapsto \Gamma(f \circ P_\zeta),
\end{equation*}
denote the continuous slice functional.
We can regard these functionals as densely defined functionals on $N^+(\mathbb{D})$.
The proof of the following lemma is an adaptation of the proof of \cite[Theorem 1.7]{mcc92a}.

\begin{lem}
  \label{lem:slice_equicont_decay}
  Let $\Gamma = \langle \cdot,h \rangle_{H^2_d}$ be a continous linear functional on $H^2_d$
  and let $h = \sum_{n=0}^\infty h_n$ be the homogeneous decomposition of $h$.
  If the slice functionals $\{\Gamma_\zeta: \zeta \in \partial \mathbb{B}_d\}$
  are equicontinuous on $(N^+(\mathbb{D}), \rho)$,
  then there exists $c > 0$ so that $\|h_n\|_{\infty} = O(e^{-c \sqrt{n}})$.
\end{lem}

\begin{proof}
  If $f = \sum_{n=0}^\infty \widehat f(n) z^n \in H^2$, then
  \beq
    \Gamma_\zeta(f) &\ = \ & \langle f \circ P_\zeta, h \rangle
    \\ &=& \sum_{n=0}^\infty \widehat f(n) \langle  \langle z,\zeta \rangle^n , h_n \rangle \\
    & =& \sum_{n=0}^\infty \widehat f(n) \ol{h_n(\zeta)}.
  \eeq
  Since the functionals $\{\Gamma_\zeta: \zeta \in \partial \mathbb{B}_d\}$ are equicontinuous
  with respect to $\rho$, there exists $\delta > 0$ such that if $f \in H^2$ and
  $\rho(f,0) < \delta$, then $|\Gamma_\zeta(f)| \le 1$ for all $\zeta \in \partial \bB_d$.
  Thus, if $f \in H^2$ and $\rho(f,0) < \delta$, then for all $n \in \bN$ and all $\zeta \in \partial \bB_d$,
  \begin{equation}
    \label{eqn:taylor_estimate}
    |\widehat f(n) h_n(\zeta)| = \Big| \int_{0}^{2 \pi} \Gamma_{e^{ it} \zeta} (f) e^{i n t} \, d t \Big|
    \le 1.
  \end{equation}

Let $\phic$ be as in \eqref{eqb7}.
By the first assertion of Lemma \ref{lemb5},
we may choose a null sequence $(\varepsilon_n)$ of real numbers so that $|\widehat \phic(n)| \ge e^{c \sqrt{n} (1 - \varepsilon_n)}$. By the second part,
  there exists $c > 0$ such that $\rho(\phic(rz),0) < \delta$ for all $r \in (0,1)$. Applying
  \eqref{eqn:taylor_estimate} to the functions $\phic(rz)$, we therefore find that
  for all $\zeta \in \partial \bB_d$,  $r \in (0,1)$ and $n \in \bN$,
  \begin{equation*}
    r^n e^{c \sqrt{n} (1 - \varepsilon_n)} |h_n(\zeta)| \le r^n | \widehat \phic(n) h_n(\zeta)| \le 1.
  \end{equation*}
  Taking the limit $r \to 1$, we see that for all $\zeta \in \partial \bB_d$,
  \begin{equation*}
    |h_n(\zeta)| \le e^{- c \sqrt{n} (1 - \varepsilon_n)},
  \end{equation*}
  so $\|h_n\|_\infty = O(e^{-c \sqrt{n}})$, for some slightly smaller $c > 0$.
\end{proof}

We can now prove Theorem \ref{thmAB3}, which we restate in equivalent form.

\begin{theorem}
  \label{thm:dual_uniform_smirnov}
  The functional
  \begin{equation*}
    \Gamma: f \mapsto \sum_{\alpha \in \mathbb{N}^d} \omega_\alpha^2 \widehat{f}(\alpha)  \overline{\widehat{h}(\alpha)}
  \end{equation*}
  is in $(N_u^+, \rho)^*$ if and only if there exist constants $M, c > 0$ so that
  \begin{equation}
    \label{eqn:dual_suff}
    |\widehat{h}(\alpha)| \omega_\alpha \le M e^{-c \sqrt{|\alpha|}}.
  \end{equation}
\end{theorem}

\begin{proof}
  (Sufficiency) Suppose $\widehat{h}(\alpha)$ satisfies \eqref{eqn:dual_suff} and let $h = \sum_{\alpha} \widehat{h}(\alpha) z^\alpha \in H^2_d$.
  We have to show that $\Gamma(f) = \langle f, h \rangle_{H^2_d}$ defines a continuous linear functional on $(N^+_u,\rho)$.
  If $h_n$ denotes the homogeneous component of degree $n$ of $h$, then Corollary
  \ref{cor:decay_equiv} implies that $\|h_n\|_{H^2_d} = O(e^{-c \sqrt{n}})$ (for potentially different $c > 0$).
  Let $f \in N_u^+$ with homogeneous decomposition $f = \sum_{n=0}^\infty f_n$. Lemma \ref{lem:Smirnov_growth}
  shows that $\|f_n\|_{H^2_d} = O(e^{c/2 \sqrt{n}})$. Hence by the Cauchy--Schwarz inequality,
  the series $\sum_{n=0}^\infty \langle f_n, h_n \rangle$
  converges absolutely. For each $N \in \mathbb{N}$, the linear functional
  $\Gamma_N(f) = \sum_{n=0}^N \langle f_n, h_n \rangle$ is continuous on $(N^+_u, \rho)$ by \eqref{eqaa6},
  so the uniform boundedness principle for F-spaces (see, e.g.\ \cite[Theorem 2.8]{Rudin91}) shows that $\Gamma(f) = \lim_{N \to \infty} \Gamma_N(f)$
  defines a continuous linear functional on $(N^+_u,\rho)$.

  (Necessity) Let $\Gamma$ be a continuous linear functional on $(N^+_u,\rho)$.
  Lemma \ref{lem:DA_Smirnov_inclusion} shows that $\Gamma$ is continuous on $H^2_d$, so there exists $h \in H^2_d$ with
  $\Gamma(f) = \langle f,h \rangle$ for all $f \in H^2_d$.
  Note that if $f \in N^+(\mathbb{D})$ and $\zeta \in \partial \mathbb{B}_d$, then $f \circ P_\zeta \in N^+_u$ with
  $\vertiii{f \circ P_\zeta} = \vertiii{f}$.
  This shows that the slice functionals $\{\Gamma_\zeta: \zeta \in \partial \mathbb{B}_d\}$ are equicontinuous
  on $N^+(\mathbb{D})$, so we deduce from Lemma \ref{lem:slice_equicont_decay} that $\|h_n\|_\infty = O(e^{-c \sqrt{n}})$ for some $c > 0$.
  Now Corollary \ref{cor:decay_equiv} yields \eqref{eqn:dual_suff}.
\end{proof}

\section{The range of adjoints of cyclic multiplication operators}
\label{secd}

First we show that Helson's theorem is true for $N^+(\htd)$.

\begin{theorem}
\label{thmD1}
Let $m \in {\rm Mult}(H^2_d)$ by cyclic.
The function $h$ is in the range of $T_m^*$ if and only if the functional
\be
\label{eqd1}
f \ \mapsto \ \sum \omat  \widehat f(\alpha) \overline{\widehat h (\alpha) } = \langle f, h \rangle_{H^2_d}
\ee
extends from $H^2_d$ to be
 a continuous linear functional on $\frac{1}{m} H^2_d$.
\end{theorem}

\begin{proof}
Suppose $h = T_m^* g$. We wish to show that for some constant $M$,
\be
\label{eqd2}
|\langle f,h \rangle_{H^2_d}|  \leq M \ \| m f \|_{H^2_d} .
\ee
But
\begin{equation*}
  \langle f,h \rangle = \langle f, T_m^* g \rangle = \langle mf ,g \rangle,
\end{equation*}
so \eqref{eqd2} holds with $M = \| g \|_{H^2_d}$.

Conversely, suppose \eqref{eqd2} holds.
Then
\[
| \langle f , h \rangle_{H^2_d} | \ \leq \ M  \| f \|_ { \frac{1}{m} H^2_d} ,
\]
for all $f$ in the dense set $\htd$ in $\frac{1}{m} \htd$.
By the Riesz representation theorem there exists a function $k = g/m$ in $ \frac{1}{m} H^2_d$ so that
\beq
 \langle f , h \rangle_{H^2_d} & \= &  \langle f , k \rangle_{\frac{1}{m} H^2_d}\\
 &=& \langle m f , mk \rangle_{H^2_d} \\
 &=&  \langle f , T_m^* g \rangle_{H^2_d}
 \eeq
 for all $f \in H^2_d$. Hence $h = T_m^* g$.
\end{proof}

By definition of the inductive limit topology, a linear functional on $(N^+(H^2_d),\mathcal{I})$ is continuous if and only if its restriction
to $\frac{1}{m} H^2_d$ is continuous for each cyclic multiplier $m$.
Hence, as a corollary we get:

 \vs
{\bf \noindent Theorem~\ref{thmAB1}.}
{\em
The function $h$ in $H^2_d$ is in
$\bigcap_{m\ {\rm cyclic\ multiplier}} {\rm ran}(T_m^*)$ if and only if
the functional \eqref{eqd1}
extends from $H^2_d$ to be
 a continuous linear functional on $(N^+(H^2_d), \I)$.
}

\vs
We have shown in Theorem~\ref{thmAB1} that conditions (ii) and (iii) are equivalent in
Theorem \ref{thmH2}, and in Theorem \ref{thm:dual_uniform_smirnov} that conditions (i) and (iv) are equivalent.
We can now also prove that condition (iv) is sufficient for (ii) to hold.

\vs
\begin{proof}[Proof of (iv) $\Rightarrow$ (ii) in Theorem \ref{thmH2}]

Suppose
$h(z) = \sum_{\alpha} \widehat{h} (\alpha) z^\alpha$ is in $H^2_d$
and for some $c > 0$ we have
\[
\widehat{h} (\alpha)  \= O \left( \frac{1}{\oma}\  e^{-c \sqrt{|\alpha|}} \right).
\]
By Theorem \ref{thm:dual_uniform_smirnov}, the functional $\Gamma: f \mapsto \langle f,h \rangle_{H^2_d}$
extends to a continuous functional on $(N^+_u,\rho)$. We showed in Theorem \ref{thmH1}
that for each cyclic multiplier $m$ of $H^2_d$, the space $\frac{1}{m} H^2_d$ is continuously contained
in $(N_u^+,\rho)$. Hence $\Gamma$ is continuous on $\frac{1}{m} H^2_d$ for each cyclic multiplier $m$,
so $\Gamma$ is continuous on $(N^+(H^2_d),\mathcal{I})$.
\end{proof}

\section{\texorpdfstring{Proof of necessity of \eqref{eqAB5}}{Proof of necessity}}
\label{secnec}

Recall that $P_\zeta(z) = \langle z,\zeta \rangle$.
To prove that (ii) implies (iv) in Theorem \ref{thmH2}, we will analyze continuous functionals $\Gamma$ on $(N^+(H^2_d),\mathcal{I})$
by studying the slice functionals
$\Gamma_\zeta$ on $N^+(\mathbb{D})$, defined by $\Gamma_\zeta(f) = \Gamma(f \circ P_\zeta)$.

First, we observe that outer functions in $H^\infty$ lift to cyclic mulipliers of $H^2_d$.

\begin{lemma}
\label{lemd10} If $\phi$ is in  $H^\infty(\D)$, then for every $\zeta \in \partial \B_d$, the function
$\phi \circ P_\zeta$ is in  ${\rm Mult}(H^2_d)$, and $\|\phi \circ P_\zeta\|_{\Mult(H^2_d)} \le \|\phi\|_\infty$.
If $\phi$ is outer, then
$\phi \circ P_\zeta$ is cyclic.
\end{lemma}
\begin{proof}
Multiplication by $\la w , \zeta \ra$ is a contraction, so by von Neumann's inequality, for every $ 0 < t < 1$,
the norm of multiplication by $\phi( \la w , t\zeta \ra)$ is
bounded by the norm of $\phi$ in $H^\infty$. As $t \uparrow 1$, these multiplication operators converge in
the weak-$*$ topology to multiplication by $\phi( \la w , \zeta \ra)$.

Now suppose $\phi$ is outer. Then there are polynomials $p_n \in \C[z]$ so that
$\phi (z) p_n(z) \to 1$ in $H^2$.
Therefore $\phi( \la w , \zeta \ra) p_n ( \la w, \zeta \ra) \to 1 $ in $H^2_d$,
so $\phi( \la w , \zeta \ra)$ is cyclic.
\end{proof}

As a consequence, we show that the Smirnov class in one variable embeds continuously into $N^+(H^2_d)$.

\begin{lem}
  \label{lem:IL_inclusion}
  For $\zeta \in \partial \mathbb{B}_d$, the map
  \begin{equation*}
    (N^+(\mathbb{D}),\I) \to (N^+(H^2_d),\I), \quad f \mapsto f \circ P_\zeta,
  \end{equation*}
  is continuous.
\end{lem}

\begin{proof}
  Let $I$ be the map in the statement of the lemma.
  Recall that by the universal property of the inductive limit topology,
  it suffices to show that for each cyclic multiplier
  $m$ of $H^2$, the restriction of $I$ to $\frac{1}{m} H^2$ is continuous.
  But $m \circ P_\zeta$ is a cyclic multiplier of $H^2_d$ by Lemma \ref{lemd10}, and the map
  \begin{equation*}
    \frac{1}{m} H^2 \to \frac{1}{m \circ P_\zeta} H^2_d, \quad f \mapsto f \circ P_\zeta,
  \end{equation*}
  is an isometry, as the embedding $H^2 \to H^2_d, f \mapsto f \circ P_\zeta$ is isometric.
  Since the inclusion $\frac{1}{m \circ P_\zeta} H^2_d \subset N^+(H^2_d)$
  is continuous, it follows that $I$ is continuous.
\end{proof}

In \cite[Theorem 2.1]{mcc90c}, it was shown that the dual spaces of $(N^+(\mathbb{D}),\mathcal{I})$ and $(N^+(\mathbb{D}),\rho)$
coincide. So as a corollary, we get:

\begin{cor}
  \label{cor:slice_continuous}
  Let $\Gamma \in (N^+(H^2_d),\mathcal{I})^*$. Then for each $\zeta \in \partial \mathbb{B}_d$,
  the slice functional $\Gamma_\zeta$ is continuous on $(N^+(\mathbb{D}),\rho)$.
\end{cor}

\begin{proof}
  Lemma \ref{lem:IL_inclusion} shows that the slice functional $\Gamma_\zeta$
  is continuous on $(N^+(\mathbb{D}),\mathcal{I})$, hence it is continuous on $(N^+(\mathbb{D}),\rho)$
  by \cite[Theorem 2.1]{mcc90c}.
\end{proof}

The key step in proving (ii) $\Rightarrow$ (iv) in Theorem \ref{thmH2} is to show that
the slice functionals $\Gamma_\zeta$ are not only individually continuous, but equicontinuous on $(N^+(\mathbb{D}),\rho)$.
To this end, we require some more preliminary results.
 For cyclic multipliers that arise as in Lemma \ref{lemd10},
Lemma \ref{lem:cyclic_Kaplansky} can be refined to conclude that $(\varphi_n)$ are cyclic.

\begin{lem}
  \label{lem:Kaplansky_outer}
  Let $\psi \in H^\infty$ be outer, let $\zeta \in \partial \mathbb{B}_d$ and let $m = \psi \circ P_\zeta$.
  Then there exists a sequence $(\varphi_n)$
  of cyclic multipliers of $H^2_d$ so that
  \begin{enumerate}
    \item $(\varphi_n m)$ converges to $1$ in the weak-$*$ topology of $\Mult(H^2_d)$, and
    \item $\|\varphi_n m \|_{\Mult(H^2_d)} \le 1$ for all $n \in \mathbb{N}$.
  \end{enumerate}
\end{lem}

\begin{proof}
  We may assume without loss of generality that $\|\psi\|_\infty = 1$.
  Since $\psi$ is outer,
  \begin{equation*}
    \psi(z) = \exp\Big( \int_{0}^{2 \pi} \frac{e^{i \theta} + z}{e^{i \theta} - z} \log|\psi(e^{i\theta})| \frac{d \theta}{2 \pi} \Big)
  \end{equation*}
  for $z \in \mathbb{D}$. Define
  \begin{equation*}
    \psi_n(z) = \exp \Big(
    \int_{0}^{2 \pi} \frac{e^{i \theta} + z}{e^{i \theta} - z} \log ( \min( n, |\psi(e^{i \theta})|^{-1})) \frac{d \theta}{2 \pi} \Big).
  \end{equation*}
  Then $\psi_n \in H^\infty$ is outer, $\|\psi_n \psi \|_\infty \le 1$ and for each $z \in \mathbb{D}$, the dominated
  convergence theorem shows that $\lim_{n \to \infty} \psi_n(z) = \frac{1}{\psi(z)}$. Define $\varphi_n = \psi_n \circ P_\zeta$.
  Then each $\varphi_n$ is a cyclic multiplier by Lemma \ref{lemd10}.
  Since $(\varphi_n m) = (\psi_n \psi) \circ P_{\zeta}$, we find that
  $\|\varphi_n m\|_{\Mult(H^2_d)} \le 1$ and that $\varphi_n m$ converges to $1$ pointwise and therefore
  in the weak-$*$ topology of $\Mult(H^2_d)$.
\end{proof}

On bounded subsets of $\Mult(\mathcal{H})$, the weak-$*$ topology coincides with the topology of pointwise convergence.
The following simple lemma shows that it sometimes suffices to establish convergence at a single point.

\begin{lem}
  \label{lem:weak-star-Gleason}
  Let $\mathcal{H}$ be a normalized complete Pick space on $X$
  and let $(\varphi_n)$ be a sequence in the closed unit ball of $\Mult(\mathcal{H})$.
  If there exists a point $x_0 \in X$ so that $\lim_{n \to \infty} \varphi_n(x_0) = 1$,
  then $(\varphi_n)$ converges to $1$ in the weak-$*$ topology of $\Mult(\mathcal{H})$.
\end{lem}

\begin{proof}
  By renormalizing the kernel (cf.\ \cite[Section 2.6]{ampi}), we may assume that the kernel is normalized at $x_0$.
  In this setting, we find that
  \begin{equation*}
    \|1 - \varphi_n\|_{\mathcal{H}}^2 = 1 + \|\varphi_n\|_{\mathcal{H}}^2 - 2 \Re \varphi_n(x_0)
    \le 2 - 2 \Re \varphi_n(x_0) \xrightarrow{n \to \infty} 0.
  \end{equation*}
  In particular, $\lim_{n \to \infty} \varphi_n(x) = 1$ for all $x \in X$.
\end{proof}

We will make use of infinite products of multipliers.
\begin{lem}
  \label{lem:infinite_product}
  Let $\mathcal{H}$ be a normalized complete Pick space on $X$ and let $(\psi_n)$ be a sequence in the closed unit ball
  of $\Mult(\mathcal{H})$. Assume that there exists $x_0 \in X$ so that $\sum_{n} | 1 - \psi_n(x_0)| < \infty$.
  \begin{enumerate}[label=\normalfont{(\alph*)}]
    \item The infinite product $\prod_{n=1}^\infty \psi_n$ converges in the weak-$*$ topology of $\Mult(\mathcal{H})$.
    \item If each $\psi_n$ is cyclic, then $\psi = \prod_{n=1}^\infty \psi_n$ is cyclic as well.
  \end{enumerate}
\end{lem}

\begin{proof}
  (a)
  For $M \ge N$, let $\Psi_{N,M} = \prod_{n=N}^M \psi_n$, which belongs to the closed unit ball of $\Mult(\mathcal{H})$.
  The assumption implies that $\Psi_{N,M}(x_0)$ converges to $1$ as $M \ge N \to \infty$,
  hence $\Psi_{N,M}(x)$ converges to $1$ as $M \ge N \to \infty$ for all $x \in X$ by Lemma \ref{lem:weak-star-Gleason}.
  From this, we deduce that the partial products $\prod_{n=1}^N \psi_n$ converge pointwise on $X$ and hence
  in the weak-$*$ topology of $\Mult(\mathcal{H})$.

(b) Since $\psi_1$ is cyclic, there is a sequence $(f_k)$ in $\mathcal{H}$
so that $\psi_1 f_k \to 1$. Therefore $\prod_{n=2}^\infty \psi_n \in [ \psi]$.
Iterating, we get for every $N$ that $\prod_{n=N}^\infty \psi_n \in [ \psi]$.
As $N \to \infty$,  $\prod_{n=N}^\infty \psi_n$ converges weakly to $1$ in  $\mathcal H$ by part (a),
so $\psi$ is cyclic.
\end{proof}

\begin{rem}
  It is in general not true that non-zero weak-$*$ limits of cyclic multipliers are cyclic. For instance,
  if $f \in H^\infty$ is a singular inner function, then $f_r$ is outer for each $0 < r < 1$ and $f_r$ converges
  to $f$ in the weak-$*$ topology of $H^\infty$ as $r \to 1$.
\end{rem}

The following is the key lemma needed to prove equicontinuity of the slice functionals.

\begin{lem}
  \label{lem:common_cyclic_factor}
  Let $f \in N^+(\mathbb{D})$. Then the map
  \begin{equation*}
    \partial \mathbb{B}_d \to (N^+(H^2_d),\mathcal{I}), \quad \zeta \mapsto f \circ P_\zeta,
  \end{equation*}
  is continuous.
\end{lem}

\begin{proof}
  Let $(\zeta_n)$ be a sequence in $\partial \mathbb{B}_d$ that tends to $\zeta$.
  It suffices to show that a subsequence of $(f \circ P_{\zeta_n})$ converges to $f \circ P_{\zeta}$ in $(N^+(H^2_d),\mathcal{I})$.
  Let $F = f \circ P_{\zeta} \in N^+(H^2_d)$.
  It is elementary to construct a sequence $(U_n)$ of $d \times d$ unitaries tending to $I$ so that $U_n \zeta_n = \zeta$ for all $n$.

  Then $P_{\zeta} \circ U_n = P_{\zeta_n}$, so we have to show that a subsequence of $(F \circ U_n)$ tends to $F$ in $(N^+(H^2_d),\mathcal{I})$.

  We will construct a cyclic multiplier $M$ so that a subsequence of $(F \circ U_n)$ tends to $F$ in $\frac{1}{M} H^2_d$.
  Continuity of the inclusion $\frac{1}{M}H^2_d \subset (N^+(H^2_d),\mathcal{I})$ then shows that the subsequence also
  converges to $F$ in $(N^+(H^2_d),\mathcal{I})$, as desired.
  To construct $M$, write $f = \frac{g}{u}$ for some $g \in H^2$ and some outer function $u \in H^\infty$.
  Let $G = g \circ P_{\zeta}$ and $m = u \circ P_{\zeta}$, so $F = \frac{G}{m}$.
  The multiplier $M$ will essentially be an infinite product of multipliers
  of the form $m \circ U_n$, suitably corrected to ensure convergence.

  First, by Lemma \ref{lem:Kaplansky_outer}, there exists a sequence $(\varphi_n)$
  of cyclic multipliers of $H^2_d$ so that $\|\varphi_n m\|_{\Mult(H^2_d)} \le 1$ for all $n$
  and so that $| 1 - \varphi_n(0) m(0)| \le 2^{-n}$.
  Next, since $m \circ U_n$ converges to $m$ in the norm of $H^2_d$, we may pass to a subsequence of $(U_n)$ to achieve that
  \begin{equation*}
    \sum_{n=1}^\infty 2^{4 n} \|\varphi_n\|_{\Mult(H^2_d)}^2 \| m - m \circ U_n\|^2_{H^2_d} \le 1.
  \end{equation*}
  Then \cite[Theorem 1.1]{ahmrfac} yields a sequence $(\psi_n)$ in $\Mult(H^2_d)$ forming
  a contractive column multiplier and a contractive multiplier $\psi \in \Mult(H^2_d)$ with $\psi(0) = 0$
  so that
  \begin{equation*}
    2^{2 n} \|\varphi_n\|_{\Mult(H^2_d)} ( m - m \circ U_n) = \frac{\psi_n}{1 - \psi};
  \end{equation*}
  whence
  \begin{equation}
    \label{eqn:common_cyclic}
    \| ( 1- \psi) ( m - m \circ U_n) \|_{\Mult(H^2_d)} \le 2^{-2 n} \|\varphi_n\|_{\Mult(H^2_d)}^{-1}.
  \end{equation}
  Let $r_n = 1 - 2^{-n}$ and let
  \begin{equation*}
    \eta_n = \varphi_n (m \circ U_n) \frac{1 - \psi}{1 - r_n \psi},
  \end{equation*}
  which is a product of cyclic multipliers and hence cyclic
  (cyclicity of $1- \psi$ was shown in \cite[Lemma 2.3]{ahmrSm}).
  Then
  \begin{equation*}
    \| \eta_n\|_{\Mult(H^2_d)}
    \le \Big\| \varphi_n (m \circ U_n - m) \frac{1 - \psi}{1 - r_n \psi} \Big\|_{\Mult(H^2_d)}
    + \Big\|\varphi_n m \frac{1 - \psi}{1 - r_n \psi} \Big\|_{\Mult(H^2_d)}.
  \end{equation*}
  We estimate the first summand using \eqref{eqn:common_cyclic} and the second
  summand using Lemma \ref{lem:vNI} combined with the bound $\|\varphi_n m\|_{\Mult(H^2_d)} \le 1$ to find that
  \begin{equation*}
    \| \eta_n\|_{\Mult(H^2_d)} \le 2^{-2n} \frac{1}{1 - r_n} + \frac{2}{1 + r_n} \le 1+ 2^{-n+1}.
  \end{equation*}
  Let $t_n = 1 - 2^{-n+1}$. Then $\|t_n \eta_n\|_{\Mult(H^2_d)} \le 1$ and $t_n \eta_n$ is cyclic for $n \ge 2$.
  Moreover, since $\psi(0) = 0$, we have $\eta_n(0) = \varphi_n(0) m(0)$.
  Since $|1 - \varphi_n(0) m(0)| \le 2^{-n}$, it follows that $\sum_{n=2}^\infty | 1 - t_n \eta_n(0)| < \infty$.
  Thus, Lemma \ref{lem:infinite_product} shows that the infinite product of the $t_n \eta_n$ converges
  to a cyclic multiplier. Let
  \begin{equation*}
    M = m^2 \prod_{n=2}^\infty t_n \eta_n.
  \end{equation*}
  Then $M$ is a cyclic multiplier.

  We finish the proof by showing that $\lim_{n \to \infty} \|M (F \circ U_n - F)\|_{H^2_d} = 0$.
  Indeed, since $F = \frac{G}{m}$, we have
  \begin{align*}
    \| M (F \circ U_n - F)\|_{H^2_d} &\le \|m^2 \eta_n (F \circ U_n - F)\|_{H^2_d} \\
                                     &\le \Big\| \varphi_n m \frac{1-\psi}{1 - r_n \psi} \Big\|_{\Mult(H^2_d)}
                                     \| m (m \circ U_n) (F \circ U_n - F) \|_{H^2_d} \\
                                     &\le 2 \|m (G \circ U_n) - (m \circ U_n) G\|_{H^2_d},
  \end{align*}
  where in the last estimate, we used again that $\|\varphi_n m\|_{\Mult(H^2_d)} \le 1$
  and Lemma \ref{lem:vNI}. Since $G \circ U_n$ tends to $G$ in $H^2_d$ and the sequence of multipliers
  $(m \circ U_n)$ tends to $m$ in the strong operator topology,
  we conclude that $\lim_{n \to \infty} \|M (F \circ U_n - F) \|_{H^2_d} = 0$, as desired.
\end{proof}

We are ready to prove equicontinuity of the slice functionals $\Gamma_\zeta$ defined
by $\Gamma_\zeta(f) = \Gamma(f \circ P_\zeta)$.

\begin{lem}
  \label{lem:slice_equicont}
  If $\Gamma \in (N^+(H^2_d),\mathcal{I})^*$, then the functionals $\{\Gamma_{\zeta}: \zeta \in \partial \mathbb{B}_d\}$
  are equicontinuous on $(N^+(\D), \rho)$.
\end{lem}

\begin{proof}
  By Corollary \ref{cor:slice_continuous}, each $\Gamma_\zeta$ is continuous on $(N^+(\mathbb{D}),\rho)$.
  By the uniform boundedness principle for F-spaces (see, e.g.\ \cite[Theorem 2.6]{Rudin91}), it therefore suffices to show that the functionals $\Gamma_\zeta$
  are pointwise bounded.
  If $f \in N^+(\mathbb{D})$, then $\{f \circ P_{\zeta}: \zeta \in \partial \mathbb{B}_d\}$
  is a compact set in $(N^+(H^2_d),\mathcal{I})$ by Lemma \ref{lem:common_cyclic_factor} and compactness of $\partial \mathbb{B}_d$.
  Hence,
  \begin{equation*}
    \sup_{\zeta \in \partial \mathbb{B}_d} |\Gamma_{\zeta} (f)| =
    \sup_{\zeta \in \partial \mathbb{B}_d} |\Gamma({f \circ P_{\zeta}})| < \infty,
  \end{equation*}
  establishing pointwise boundedness.
\end{proof}

Finally, we can finish the proof of Theorem \ref{thmH2}.
\begin{proof}[Proof of (ii) $\Rightarrow$ (iv) in Theorem \ref{thmH2}]
If $\Gamma = \langle \cdot,h \rangle \in (N^+(H^2_d),\mathcal{I})^*$,
then the slice functionals $\{\Gamma_\zeta: \zeta \in \partial \mathbb{B}_d \}$ are equicontinuous on $(N^+(\mathbb{D}), \rho)$
by Lemma \ref{lem:slice_equicont}. Thus, an application of Lemma \ref{lem:slice_equicont_decay}
shows that the homogeneous components $h_n$ of $h$ satisfy the decay condition $\|h_n\|_\infty = O(e^{-c \sqrt{n}})$ for some $c> 0$.
From Corollary \ref{cor:decay_equiv}, we get \eqref{eqAB5}, for some $c > 0$.
\end{proof}

\section{\texorpdfstring{$N^+(\htd) $ vs. $N_u^+$}{Comparing the two Smirnov classes}}
\label{secf}

The goal of this section is to construct functions in the ball algebra $A(\mathbb{B}_d)$ (and hence in $N_u^+$)
that
do not belong to $N^+(H^2_d)$. Let $d \in \bN$ and let
\begin{equation*}
  \tau: \ol{\bB_d} \to \ol{\bD}, \quad z \mapsto d^{d/2} z_1 z_2 \ldots z_d.
\end{equation*}
It follows from the inequality of arithmetic and geometric means that $\tau$ maps $\ol{\bB_d}$
onto $\ol{\bD}$ and $\bB_d$ onto $\bD$. Let
\begin{equation*}
  a_n = \| \tau^n\|^{-2}_{H^2_d}
\end{equation*}
and let $\cK_d$ be the reproducing kernel Hilbert space on $\bD$ with reproducing kernel
\begin{equation*}
  K(z,w) = \sum_{n=0}^\infty a_n (z \ol{w})^n.
\end{equation*}
Then $\cK_d$ embeds isometrically into $H^2_d$. For a proof of the following lemma,
see \cite[Lemma 7.1]{AHM+20a}. The weighted Dirichlet space $\mathcal{D}_{1/2}$ is the reproducing
kernel Hilbert space of holomorphic functions on $\mathbb{D}$ with reproducing kernel $(1 - z \overline{w})^{-1/2}$.

\begin{lemma}
  \label{lem:embed}
  The map
  \begin{equation*}
    V: \cK_d \to H^2_d, \quad f \mapsto f \circ \tau,
  \end{equation*}
  is an isometry whose range
  is the space of all functions in $H^2_d$ that are power series in $z_1 z_2 \ldots z_d$.

  Moreover, $\mathcal{K}_2$ is the weighted Dirichlet space $\mathcal{D}_{1/2}$, and
  $\mathcal{K}_d \subset A(\mathbb{D})$ for $d \ge 4$.
\end{lemma}

The following proposition makes it possible to construct functions in $H^\infty(\bB_d)$
that do not belong to $N^+(H^2_d)$. In fact, the resulting functions are not even quotients
of multipliers with non-vanishing denominators.

\begin{proposition}
  \label{prop:not_smirnov}
  Let $f \in H^\infty(\bD)$ be a function with the property that
  \begin{equation*}
    g \cdot f \notin \cK_d \quad \text{ for all } g \in \cK_d \setminus \{0\}.
  \end{equation*}
  Then $f \circ \tau \in H^\infty(\bB_d)$, but $f \circ \tau \notin N(H^2_d)$.
\end{proposition}

We first consider an example showing that the hypothesis of the proposition can be satisfied.

\begin{example}
  \label{exa:h_infty_4}
{\rm
  Let $d \ge 4$ and let $f$ be a Blaschke product whose zeros accumulate at every point of $\mathbb T$.
  Since $\cK_d \subset A(\bD)$ by Lemma \ref{lem:embed}, we find that
  $g f \notin \cK_d$ for all $g \in \cK_d \setminus\{0\}$. Thus, $f \circ \tau \in H^\infty(\bB_d)$,
  but $f \notin N^+(H^2_d)$ by the proposition.
}
\end{example}

\begin{proof}[Proof of Proposition \ref{prop:not_smirnov}]
  Let $f \in H^\infty(\bD)$ be as in the statement of Proposition \ref{prop:not_smirnov}.
  Since $f \in H^\infty(\mathbb{D})$, we have $f \circ \tau \in H^\infty(\bB_d)$.
  Suppose towards a contradiction that  $f \circ \tau = \frac{\varphi}{\psi}$ for
  functions $\varphi,\psi \in H^2_d$ and $\psi$ non-vanishing, hence
  $\varphi = \psi (f \circ \tau)$.

  We will show that we can achieve that $\varphi$ and $\psi$ belong to the
  range of the isometry $V$ of Lemma \ref{lem:embed}.
  To this end, we decompose $\varphi = \varphi_1 + \varphi_2$
  and $\psi = \psi_1 + \psi_2$ into holomorphic functions so that $\varphi_1$ and $\psi_1$
  are power series in $z_1 z_2 \cdots z_d$ and so that no monomial
  of the form $(z_1 z_2 \ldots z_d)^n$ occurs in $\varphi_2$ or $\psi_2$.
  Since $f \circ \tau$ is a power series in $z_1 z_2 \ldots z_d$,
  it follows that
  \begin{equation}
    \label{eqn:diagonal_embedding}
    \varphi_1 = (f \circ \tau) \psi_2.
  \end{equation}
  Moreover, since $\varphi,\psi \in H^2_d$, we also have $\varphi_1,\psi_1 \in H^2_d$
  and $\psi_1(0) = \psi(0) \neq 0$.
  Hence, $\varphi_1$ and $\psi_1$ belong
  to the range of the isometry $V$ of Lemma \ref{lem:embed},
  so there exist
  $g,h \in \cK_d$ such that $\varphi_1 = h \circ \tau$ and $\psi_1 = g \circ \tau$,
  and clearly $g(0)  = \psi_1(0) \neq 0$.

  From \eqref{eqn:diagonal_embedding}, we infer that $(h \circ \tau) = (f \circ \tau) \cdot (g \circ \tau)$.
  As $\tau$ maps $\bB_d$ onto $\bD$, this means that $h = f \cdot g$, contradicting the assumption on $f$.
  Thus,  $f \circ \tau \notin N(H^2_d)$.
\end{proof}

The construction of Example \ref{exa:h_infty_4} can be refined to work for all $d \ge 2$
and to yield a function in the ball algebra $A(\bB_d)$.

\vskip 10pt

{\bf \noindent Theorem}~\ref{cor:ball_nevanlinna}
{\em
  $A(\bB_d) \not \subset N(H^2_d)$ for all natural numbers $d \ge 2$.
}

\begin{proof}
  We first consider the case $d = 2$. We will use Lemma \ref{lem:embed} to embed the weighted
  Dirichlet space $\mathcal{D}_{1/2}$ into $H^2_2$.

  A sequence $(z_n)$ in $\mathbb{D}$ is called a zero set for $\mathcal{D}_{1/2}$ if there exists
  a function $f \in \mathcal{D}_{1/2} \setminus \{0\}$ that vanishes precisely on $\{z_n: n \in \mathbb{N}\}$.
  By \cite[Theorem 2]{PP11}, there exists a Blachke sequence $(z_n)$ that is not a zero
  set for $\mathcal D_{1/2}$ and whose only cluster point is $1$.
  Let $B$ the Blaschke product with zeros $(z_n)$ and let $f(z) = B(z) (1 -z)$. Then $f \in A(\bD)$.
  Since $(z_n)$ is not a zero set for $\mathcal{D}_{1/2}$, it is a uniqueness
  set for $\mathcal{D}_{1/2}$ by \cite[Proposition 9.37]{ampi},
  meaning that any function in $\mathcal{D}_{1/2}$ vanishing on $\{z_n: n \in \mathbb{N}\}$
  is identically zero.
  Thus,
  $g \cdot f \notin \mathcal D_{1/2}$ for any $g \in \mathcal D_{1/2} \setminus\{0\}$. In this
  setting, Lemma \ref{lem:embed} and Proposition \ref{prop:not_smirnov} imply that $f \circ \tau \notin N(H^2_2)$,
  but clearly $f \circ \tau \in A(\bB_2)$.

  The case of an arbitrary natural number $d \ge 2$ is easily deduced from the case $d=2$.
  Indeed,
  let $F \in A(\bB_2)$ but $F \notin N(H^2_2)$, let
  $P: \bB_d \to \bB_2$ be the projection onto the first $2$ coordinates and define $G = F \circ P$.
  Then $G \circ P \in A(\bB_d)$. If $G$ were of the form $G = \varphi / \psi$ for $\varphi,\psi \in \Mult(H^2_d)$
  and $\psi$ non-vanishing, then
  \begin{equation*}
    F = G \big|_{\bB_2} = \frac{\varphi \big|_{\bB_2}}{\psi \big|_{\bB_2}}
  \end{equation*}
  would be a quotient of two functions in $\Mult(H^2_2)$, a contradiction. Thus, $G \notin N(H^2_d)$.
\end{proof}

Theorem \ref{cor:ball_nevanlinna} also implies the following fact about topologies
on the Smirnov class $N^+(H^2_d)$.

\begin{proposition}
  \label{lem:top_obstacle}
  Let $d \ge 2$ be a natural number and let $\tau$ be a topology on $N^+(H^2_d)$ that
  gives $N^+(H^2_d)$ the structure of a topological vector space. Suppose that
  \begin{enumerate}
    \item point evaluations at points in $\bB_d$ are $\tau$-continuous, and
    \item $(N^+(H^2_d),\tau)$ is sequentially complete.
  \end{enumerate}
  Then there exists a sequence $(f_n)$ of functions analytic in a neighborhood of
  $\ol{\bB_d}$ such that $(f_n)$ tends to zero uniformly on $\ol{\bB_d}$, but is not null
  with respect to $\tau$.
\end{proposition}

\begin{proof}
  By Theorem \ref{cor:ball_nevanlinna}, there exists $f \in A(\bB_d)$ but $f \notin N^+(H^2_d)$.
  For $n \in \bN$, let
  \begin{equation*}
    g_n(z) = f( (1-\tfrac{1}{n+1}) z),
  \end{equation*}
  so that $(g_n)$ is analytic in a neighborhood of $\ol{\bB_d}$ and
  converges to $f$ uniformly on $\ol{\bB_d}$. Since $f \notin N^+(H^2_d)$, the first assumption on $\tau$ implies
  that $(g_n)$ does not converge with respect to $\tau$, and hence it is not $\tau$-Cauchy by the second assumption.
  Therefore, there exist strictly increasing sequences $(n_k)$ and $(m_k)$ of natural numbers and a $\tau$-neighborhood $U$ of the origin such that $g_{n_k} - g_{m_k} \notin U$ for all $k \in \bN$. Let $f_k = g_{n_k} - g_{m_k}$.
  Then $(f_k)$ is not $\tau$-null, but converges to zero uniformly on $\ol{\bB_d}$.
\end{proof}

\section{The containing Fr\'echet space of \texorpdfstring{$\npu$}{the uniform Smirnov class}}
\label{secfr}

In \cite{yan73b}, Yanagihara described a Fr\'echet space $F^+$ that contained the Smirnov class, a sort of ``locally convex completion''. In this section, we shall do the same for $\npu$. Throughout this section, we shall assume that $f$ is a holomorphic function on $\bB_d$ and that $f = \sum_{n=0}^\i f_n$ is its decomposition into homogeneous polynomials.

\begin{lemma}
\label{F+Lemma}
Let $f$ be analytic on $\B_d$, and for $0\le r<1$ write $M(r,f)=\sup_{|z|=r}|f(z)|$. Then the following are equivalent:
\begin{enumerate}[label=\normalfont{(\roman*)}]
  \item $\displaystyle{\limsup_{r \to 1} (1-r)\log M(r,f)\le 0,}$

\item $\displaystyle{\limsup_{n\to \infty}\frac{\log\|f_n\|_\infty}{\sqrt{n}} \le 0,}$

\item $\displaystyle{\limsup_{n\to \infty}\frac{\log\|f_n\|_{H^2_d}}{\sqrt{n}} \le 0,}$
\item $\displaystyle{ \limsup_{|\alpha|\to \infty}\frac{\log |\hat{h}(\alpha)|\omega_\alpha}{\sqrt{|\alpha|}} \le 0.}$
\end{enumerate}
\end{lemma}
\begin{proof} The equivalence of (ii), (iii), and (iv) follows from Lemma \ref{lem:norm_equiv}. If $z\in \D$ and $w\in \B_d$, then $f(zw)=\sum_{n=0}^\infty f_n(w)z^n$. Thus the equivalence of (i) and (ii) follows from Lemma 1 of \cite{yan73b}, p. 96.
\end{proof}
Let $F^+_u$ be the collection of all analytic functions on $\B_d$ that satisfy any of the conditions of  Lemma \ref{F+Lemma}. For a positive constant $c$ and $f\in F^+_u$ set
\be
\label{eqfr1}
\|f\|^2_c= \sum_{n=0}^\infty \|f_n\|^2_{H^2_d}\ e^{-c\sqrt{n}},
\ee
 and define a locally convex topology on $F^+_u$ by the family of semi-norms $\{\| \cdot \|_c\}_{c>0}$. Since $\|f\|^2_c$ is increasing as $c$ decreases to 0, it is clear that $F^+_u$ is metrizable, and a translation invariant metric is given by
 \be
 \label{eqfr2}
 \dF(f,g) \= \sum_{k=1}^\infty 2^{-k} \frac{ \|f-g\|_{1/k}}{1+\|f-g\|_{1/k}}.
 \ee

Furthermore, one easily checks that for each $f=\sum_{n\ge 0}f_n$, the partial sums $p_n=\sum_{k=0}^n f_k$ converge to $f$ in $\|\cdot\|_c$ for each $c>0$, hence the polynomials are dense in $F^+_u$. Moreover, a routine argument shows that $F^+_u$ is complete with respect to the metric $\dF$. Therefore  $F^+_u$ is a Fr\'echet space.

\begin{prop}
  \label{prop:frechet_contained}
  $(N^+_u,\rho)$ embeds continuously into $F^+_u$.
\end{prop}

\begin{proof}
  This is immediate from Lemma \ref{lem:Smirnov_growth} and the closed graph theorem for $F$-spaces,
  see, e.g.\ \cite[Theorem 2.15]{Rudin91}.
\end{proof}

The next result, combined with Theorem \ref{thm:dual_uniform_smirnov} and Corollary \ref{cor:decay_equiv},
shows that $F_u^+$ has the same dual space as $N_u^+$.

\begin{prop}
\label{propfr1}
 Let $h$ be analytic on $\mathbb B_d$, let $h=\sum_{n\ge 0}h_n$ be its expansion into homogeneous polynomials, and assume that there are $M, c>0$ such that $\|h_n\|_{H^2_d} \le M e^{-c\sqrt{n}}$ for each $n\in \mathbb N$. Then for each $f\in F^+_u$, the series
$$\Gamma_h(f)=\sum_{n\ge 0} \la f_n,h_n\ra_{H^2_d}= \sum_{n=0}^\infty\sum_{|\alpha|=n }\omega_\alpha^2 \hat{f}(\alpha)\overline{\hat{h}(\alpha)}$$
converges absolutely, and $\Gamma_h$ defines a continuous linear functional on $F^+_u$.

Moreover, every continuous linear functional on $F^+_u$ is of this form.
\end{prop}
\begin{proof} For each $n \in \mathbb N$ the Cauchy--Schwarz inequality implies that
$$\sum_{|\alpha|=n }\omega_\alpha^2 |\hat{f}(\alpha)\hat{h}(\alpha)| \le \|f_n\|_{H^2_d} \|h_n\|_{H^2_d} \le M \|f_n\|_{H^2_d} e^{-c\sqrt{n}}.$$ Hence, by another application of the Cauchy--Schwarz inequality, we see that
 $$\sum_{n=0}^\infty \sum_{|\alpha|=n }\omega_\alpha^2 |\hat{f}(\alpha)\hat{h}(\alpha)| \ \le \ M \|f\|_c \ \sqrt{\sum_{n=0}^\infty e^{-c\sqrt{n}}}$$ holds for each $f\in F^+_u$. This  implies that $\Gamma_h$ is continuous on $F^+_u$.

 The converse follows from the continuous inclusion $(N_u^+,\rho) \subset F_u^+$ in Proposition \ref{prop:frechet_contained} and our earlier description of the dual space of $(N_u^+,\rho)$ in Theorem \ref{thm:dual_uniform_smirnov}.
 We provide a direct argument.

If $\Phi$ is a continuous linear functional on $F^+_u$,
then it follows from the definition of the topology on $F^+_u$ in terms of the seminorms $\|\cdot\|_c$ that there exist $c > 0$ and $M \ge 0$ so that $|\Phi(f)| \le M \|f\|_c$
for every $f \in F_u^+$.
To see this from the metric $\dF$, notice that by continuity of $\Phi$, there is $\delta>0$ such that $|\Phi(f)|<1$, whenever $f \in F^+_u$ with $\dF(f,0)<\delta$. Choose $N$ such that $2^{-N}<\delta$ and set $\varepsilon= \frac{\delta}{2N}$. Then for each $f\in F^+_u$ with $\|f\|_{1/N} <\varepsilon$ we have
\beq
\dF(f,0) & \ \le \ & \sum_{k=1}^N 2^{-k} \|f\|_{1/k} + 2^{-N} \\
&<& N \|f\|_{1/N} +\delta/2\\
&<& \delta.
\eeq
 This implies that $|\Phi(f)| \le \frac{1}{\varepsilon} \|f\|_{1/N}$ for every polynomial $f$.
 Let $\mathcal H_N$ denote the Hilbert space of holomorphic functions on $\B_d$ with norm given
 by \eqref{eqfr1} with $c = \frac{1}{N}$.
 Then
 $\Phi$ extends to define a bounded linear functional on  $\mathcal H_N$.
  By the Riesz representation theorem, there is $g=\sum_{n=0}^\infty g_n\in \mathcal H_N$ such that $\|g\|_{1/N}\le 1/\varepsilon$ and \[
  \Phi(f)= \sum_{n=0}^\infty e^{-\frac{\sqrt{n}}{N}} \la f_n,g_n\ra_{H^2_d}.\]
   For $n \in \mathbb N$ set $h_n= e^{-\frac{\sqrt{n}}{N}}g_n$, then
\beq
\|h_n\|_{H^2_d} &\=& e^{-\frac{\sqrt{n}}{N}}\|g_n\|_{H^2_d} \\
&\le &  e^{-\frac{\sqrt{n}}{2N}}\|g\|_{1/N}\\
&\le&  \frac{e^{-\frac{\sqrt{n}}{2N}}}{\varepsilon},
\eeq
and $\Phi=\Gamma_h$.
\end{proof}

From Proposition \ref{prop:frechet_contained}, we get that $\npu$ embeds continuously into $F^+_u$,
and by Theorem \ref{thmH1}, so does $\frac{1}{m} H^2_d$ for each cyclic multiplier $m$.
Since $F_u^+$ is locally convex, it follows that $(\nph, \I)$ embeds continuously into $F_u^+$.
So if we use the metric $\dF$ from \eqref{eqfr2}, we get that the identity map is continuous from
$(\nph, \I)$ to $(\nph, \dF)$. Moreover, since the polynomials are dense in $F^+_u$, we have that $\nph$ is dense
in $F^+_u$. Therefore $(\nph, \dF)$ has the same dual as $F^+_u$, and by Theorem \ref{thmH2} and Proposition \ref{propfr1}, this is the
same as the dual of $(\nph, \I)$.

If $\mathcal V$ is a locally convex vector space with dual ${\mathcal V}^*$, there is a finest
topology on $\mathcal V$ with this dual, called the Mackey topology for the pair.
 Certain conditions on a topology force it to coincide with the Mackey topology; these include being metrizable (as $\dF$ is) or being barrelled (as $\I$ is) \cite[10-1.9]{wil}.
Putting these facts together, we get the following conclusion.

\bt
\label{thmfr1}
On $\nph$, the inductive limit topology $\I$ and the metric topology $\dF$ coincide.
\et
For the convenience of the reader, here is a direct proof.

\begin{proof}
As the  by Theorem \ref{thmH2} and Proposition \ref{propfr1} identity map is continuous from
$(\nph, \I)$ to $(\nph, \dF)$ by the discussion preceding the theorem, we need to prove that the identity map from
 $(\nph, \dF)$ to $(\nph,\I)$ is also continuous. To do this,  it suffices to show that for each continuous semi-norm $p$ on $(\nph,\I)$, the inclusion $(N^+,\dF)$ to $(\nph,p)$ is continuous, since the topology on any locally convex space is give by its continuous semi-norms (see e.g. \cite[1.10.1]{ed95}).

 So let $(f_n)$ be a sequence in $(\nph,\dF)$ tending to zero. We claim that $(f_n)$ tends to zero weakly in $(\nph,p)$. Indeed, if $\Gamma$ is a continuous linear functional on $(\nph,p)$, then $\Gamma$ is continuous on $(\nph,\I)$ and hence on $(\nph,\dF)$ by Theorem \ref{thmH2} and Proposition \ref{propfr1}.
 Thus, $\Gamma(f_n)$ tends to $0$ for all $\Gamma$, establishing weak convergence. The uniform boundedness principle for the semi-normed space $(\nph,p)$ implies that $\sup_n p(f_n) < \infty$. (This follows from the uniform boundedness principle for normed spaces by taking a quotient.) So we have shown that if $\dF(f_n,0) \to 0$, then $p(f_n)$ is bounded. But if $\dF(f_n,0) \to 0$, then by passing to a subsequence, we may achieve that $\dF(n
        f_n,0) \to 0$, hence $p(n f_n)$ is bounded, so $(f_n)$ tends to $0$ in $(\nph,p)$. This proves continuity of $(\nph,\dF) \to (\nph,p)$.
      \end{proof}

\section{Open Questions}

Lemma \ref{lem:Kaplansky_outer} says that if $m$ is the lift of an outer function, then one
can find cyclic multipliers $\varphi_n$ so that $\varphi_n m$ is in the unit ball and  tends to $1$ weak-*.
Is this true for all cyclic $m$?
\begin{question}
\label{q1}
Can Lemma \ref{lem:cyclic_Kaplansky} be improved to conclude that each $\varphi_n$ can be chosen to be cyclic?
\end{question}

As pointed out in Section \ref{secintro}, when $d = 1$ there is no difference between the common range
of co-analytic Toeplitz operators with cyclic symbols, and all non-zero co-analytic Toeplitz operators.
Is this true for $d \geq 2$?
\begin{question}
\label{q2}
What is $\bigcap \{ {\rm ran}(\tms) : m \in {\rm Mult}(\htd), m \neq 0 \}$?
\end{question}


In \cite{naw89}, M. Nawrocki showed that the dual of the full Smirnov class on the ball $\B_d$ consists of those
functionals $\Gamma$ satisfying
\be
\label{eqh1}
| \Gamma ( z^\alpha) | \ \leq \ M \oma e^{- c |\alpha|^{d/(d+1)}}.
\ee
In \cite{kormcc96}, it is shown that Condition \eqref{eqh1} is not necessary for a function
to be in the common range of all co-analytic Toeplitz operators on the Hardy space of the sphere.
\begin{question}
\label{q3}
What is the common range of all non-zero co-analytic Toeplitz operators on
$H^2(\partial \B_d)$?
\end{question}

\begin{question}
Do Questions \ref{q2} and \ref{q3} have the same answer?
\end{question}

\appendix
\section{Inductive limits}
\label{secil}

Let $\mathcal V$ be a  topological vector space, and for each $\alpha$ in some index set $A$ let
$X_\alpha$ be a topological vector space that embeds continuously in $\mathcal V$. For example, $\mathcal V$ could be all holomorphic functions on some domain $\Omega$, with the topology of locally uniform convergence, and each $X_\alpha$ could be some Banach space of holomorphic functions on $\Omega$.

Let $Y = \cup_{\alpha \in A} X_\alpha$. There are two different ways to define an inductive limit topology on $Y$. The first, which we shall call the locally convex inductive limit topology and denote $\IH$,
has as a neighborhood basis at $0$ all absolutely convex sets $U$ with the property that $U \cap X_\alpha$ is open
for every $\alpha \in A$. A basis for the topology is sets of the form $y + U$, where $y \in Y$ and $U$
is a neighborhood of $0$.
The topology $\I$ has advantages:
\begin{itemize}
\item
It is a vector space topology (addition and scalar multiplication are continuous).
\item
The TVS space $(Y, \IH)$ is locally convex.
\end{itemize}
But
\begin{itemize}
\item
It can be hard to describe what a general open set is.
\item
The topology may be trivial, with $Y$ the only non-empty open set (see Example \ref{exj1}).
\end{itemize}

A second topology, which we shall call the non-locally convex inductive limit topology and denote $\II$,
has as a basis all sets $U$ with the property that $U \cap X_\alpha$ is open
for every $\alpha \in A$.
In general, the topologies $\IH$ and $\II$ are distinct, though G. Edgar points out in
\cite{ed73} that even Choquet did not realize this.
Topology $\II$ has advantages
\begin{itemize}
\item
It is much simpler to describe the open sets.
\item
It frequently turns out to be given by a complete metric (eg. for $N^+(\D)$ \cite{mcc90c}, or Example \ref{exj1}).
\end{itemize}
However
\begin{itemize}
\item
It may not be a vector space topology --- addition may not be continuous \cite{ed73}, where the example
is all continuous functions with compact support on $\R$, viewed as the union of the continuous functions
supported on $[-n,n]$.
\end{itemize}

It is immediate that $ (Y, \II)$ embeds continuously in $(Y, \IH)$.
Moreover, if $\Gamma$ is a linear functional on $Y$, one can check its continuity just by looking at the
preimage of convex open sets in $\C$, so the continuous linear functionals are the same on both spaces:
\[
(Y, \II)^* \= (Y, \IH)^*.
\]

\begin{example}
{\rm
\label{exj1}
An illuminating example is to
 let
$\mathcal V$ denote the space of all measurable functions on $[0,1]$, identifying those that are zero a.e.
with respect to Lebesgue measure.
For $w$ any integrable function that is positive a.e., let $L^2(w) = \{ f : \int |f|^2 w < \infty \}$.
Then $\mathcal V = \cup_{w} L^2(w)$.
The topology $\IH$ is the trivial topology, and the topology $\II$ is the topology of convergence in measure,
which is given by a complete metric.


Indeed, let us define a metric on $\mathcal V$ by
\[
\dL(f,g) \= \int_0^1 \min(1, |f(x) - g(x) |) dx .
\]
This is complete, and gives the topology of convergence in measure.
It can be shown that the topologies given by $\dL$ and  $\II$  coincide.

%
%
%
%
%
}
\end{example}

\section{Smirnov class on the disk}
\label{secend}

On the disk, the following three definitions
of the Smirnov class $N^+(\mathbb{D})$ are equivalent.

\begin{proposition}
\label{prl1}
  Let $f \in N(\mathbb{D})$. The following assertions are equivalent.
  \begin{enumerate}[label=\normalfont{(\roman*)}]
    \item $\lim_{r \to 1} \int_{0}^{2 \pi} \log^+ |f (r e^{ i \theta})| \frac{d \theta}{2 \pi}
   = \int_{0}^{2 \pi} \log^+ |f ( e^{ i \theta})| \frac{d \theta}{2 \pi}$.
 \item $\lim_{r \to 1} \int_{0}^{2 \pi} \log(1 + |f(r e^{i \theta}) - f(e^{i \theta})|) \frac{d \theta}{2 \pi} = 0$.
 \item $\lim_{r \to 1} \vertiii{f - f_r} = 0$.
  \end{enumerate}
\end{proposition}

\begin{proof}
  (i) $\Rightarrow$ (iii)
  It is known that (i) implies that $\{\log^+ |f_r| : 0 < r < 1\}$ is uniformly integrable
  over $\mathbb{T}$; see \cite[Theorem A.3.7]{EKM+14}.
  Hence $\{\log( 1 + |f_r|) : 0 < r < 1 \}$ is uniformly integrable.
  Moreover, if $g \in \mathcal{O}(\mathbb{D})$, then $\log(1 + |g|)$ is subharmonic
  because $t \mapsto \log(1 + e^t)$ is increasing and convex. Thus,
  by Vitali's convergence theorem, we find that
  \begin{align*}
    \vertiii{f - f_r} &=
    \lim_{s \to 1} \int_{0}^{2 \pi} \log (1 + |f(s e^{i \theta}) - f (r s e^{ i\theta})|) \frac{d \theta}{2 \pi} \\
                      &= \int_{0}^{2 \pi} \log (1 + |f( e^{i \theta}) - f (r e^{ i\theta})|) \frac{d \theta}{2 \pi}.
  \end{align*}
  Applying Vitali's convergence theorem again, we see that the last expression tends to $0$ as $r \to 1$, so (iii) holds.

  (iii) $\Rightarrow$ (ii) By Fatou's lemma,
  \begin{align*}
    &\quad\int_{0}^{2 \pi} \log(1 + |f(r e^{i \theta}) - f(e^{i \theta})|) \frac{d \theta}{2 \pi} \\
    &\le
    \sup_{0 < s < 1}
 \int_{0}^{2 \pi} \log(1 + |f(s r e^{i \theta}) - f(s e^{i \theta})|) \frac{d \theta}{2 \pi} \\
                                                                                            &= \vertiii{f - f_r}
                                                                                            \xrightarrow{r \to 1} 0,
  \end{align*}
  so (ii) holds.

  (ii) $\Rightarrow$ (i) From the elementary inequality $| \log^+ a - \log^+ b| \le \log(1 + |a - b|)$ for $a,b \ge 0$,
  we deduce that
  \begin{equation*}
    \int_{0}^{2 \pi} | \log^+ |f(r e^{i \theta})| - \log^+ |f( e^{i \theta})| | \frac{d \theta}{2 \pi}
    \le \int_{0}^{2 \pi} \log( 1 + | f(r e^{ i \theta}) - f(e^{ i \theta})|) \frac{d \theta}{ 2\pi},
  \end{equation*}
  so (ii) implies that $\log^+ |f_r|$ tends to $\log^+|f|$ as $r \to 1$ in $L^1(\mathbb{T})$.
  In particular, (i) holds.
\end{proof}

\black
\bibliography{commonrange}

\begin{bibdiv}
\begin{biblist}

\bib{agmc_cnp}{article}{
      author={Agler, Jim},
      author={McCarthy, John~E.},
       title={Complete {Nevanlinna-Pick} kernels},
        date={2000},
     journal={J. Funct. Anal.},
      volume={175},
      number={1},
       pages={111\ndash 124},
}

\bib{ampi}{book}{
      author={Agler, Jim},
      author={M\raise.45ex\hbox{c}Carthy, John~E.},
       title={Pick interpolation and {H}ilbert function spaces},
      series={Graduate Studies in Mathematics},
   publisher={American Mathematical Society, Providence, RI},
        date={2002},
      volume={44},
        ISBN={0-8218-2898-3},
         url={http://dx.doi.org/10.1090/gsm/044},
      review={\MR{1882259}},
}

\bib{ahmrSm}{article}{
      author={Aleman, Alexandru},
      author={Hartz, Michael},
      author={M\raise.45ex\hbox{c}Carthy, John~E.},
      author={Richter, Stefan},
       title={The {S}mirnov class for spaces with the complete {P}ick
  property},
        date={2017},
        ISSN={0024-6107},
     journal={J. Lond. Math. Soc. (2)},
      volume={96},
      number={1},
       pages={228\ndash 242},
         url={https://doi.org/10.1112/jlms.12060},
      review={\MR{3687947}},
}

\bib{ahmrfac}{article}{
      author={Aleman, Alexandru},
      author={Hartz, Michael},
      author={M\raise.45ex\hbox{c}Carthy, John~E.},
      author={Richter, Stefan},
       title={Factorizations induced by complete {N}evanlinna-{P}ick factors},
        date={2018},
        ISSN={0001-8708},
     journal={Adv. Math.},
      volume={335},
       pages={372\ndash 404},
         url={https://doi.org/10.1016/j.aim.2018.07.013},
      review={\MR{3836670}},
}

\bib{AHM+20a}{article}{
      author={Aleman, Alexandru},
      author={Hartz, Michael},
      author={M\raise.45ex\hbox{c}Carthy, John~E.},
      author={Richter, Stefan},
       title={Multiplier tests and subhomogeneity of multiplier algebras},
        date={2020-08-03},
     journal={arXiv:2008.00981},
      eprint={http://arxiv.org/abs/2008.00981v1},
}

\bib{Arveson98}{article}{
      author={Arveson, William},
       title={Subalgebras of {$C\sp *$}-algebras. {III}. {M}ultivariable
  operator theory},
        date={1998},
        ISSN={0001-5962},
     journal={Acta Math.},
      volume={181},
      number={2},
       pages={159\ndash 228},
      review={\MR{1668582 (2000e:47013)}},
}

\bib{bfv}{unpublished}{
      author={Borichev, A.},
      author={Frank, R.},
      author={Volberg, A.},
       title={Counting eigenvalues of {Schr\"odinger} operator with complex
  fast decreasing potential},
        note={arXiv:1811.05591},
}

\bib{drs15}{article}{
      author={Davidson, Kenneth~R.},
      author={Ramsey, Christopher},
      author={Shalit, Orr~Moshe},
       title={Operator algebras for analytic varieties},
        date={2015},
        ISSN={0002-9947},
     journal={Trans. Amer. Math. Soc.},
      volume={367},
      number={2},
       pages={1121\ndash 1150},
         url={https://doi.org/10.1090/S0002-9947-2014-05888-1},
      review={\MR{3280039}},
}

\bib{davmcc}{article}{
      author={Davis, B.M.},
      author={M\raise.45ex\hbox{c}Carthy, J.E.},
       title={Multipliers of de {Branges} spaces},
        date={1991},
     journal={Michigan Math. J.},
      volume={38},
       pages={225\ndash 240},
}

\bib{ed73}{article}{
      author={Edgar, G.~A.},
       title={Two function-space topologies},
        date={1973},
        ISSN={0002-9939},
     journal={Proc. Amer. Math. Soc.},
      volume={39},
       pages={219\ndash 220},
         url={https://doi-org.libproxy.wustl.edu/10.2307/2039022},
      review={\MR{313771}},
}

\bib{ed95}{book}{
      author={Edwards, R.E.},
       title={Functional analysis theory and applications},
   publisher={Dover},
     address={New York},
        date={1995},
}

\bib{EKM+14}{book}{
      author={El-Fallah, Omar},
      author={Kellay, Karim},
      author={Mashreghi, Javad},
      author={Ransford, Thomas},
       title={A primer on the {D}irichlet space},
      series={Cambridge Tracts in Mathematics},
   publisher={Cambridge University Press, Cambridge},
        date={2014},
      volume={203},
        ISBN={978-1-107-04752-5},
      review={\MR{3185375}},
}

\bib{gar81}{book}{
      author={Garnett, John~B.},
       title={Bounded analytic functions},
   publisher={Academic Press},
     address={New York},
        date={1981},
}

\bib{Hartz17a}{article}{
      author={Hartz, Michael},
       title={On the isomorphism problem for multiplier algebras of
  {N}evanlinna-{P}ick spaces},
        date={2017},
        ISSN={0008-414X},
     journal={Canad. J. Math.},
      volume={69},
      number={1},
       pages={54\ndash 106},
         url={https://doi.org/10.4153/CJM-2015-050-6},
      review={\MR{3589854}},
}

\bib{hel90}{incollection}{
      author={Helson, H.},
       title={Large analytic functions {II}},
        date={1990},
   booktitle={Analysis and partial differential equations},
      editor={Sadosky, Cora},
   publisher={Marcel Dekker},
     address={Basel},
}

\bib{kormcc96}{article}{
      author={Korenblum, B.~I.},
      author={M\raise.45ex\hbox{c}Carthy, J.E.},
       title={The range of {Toeplitz} operators on the ball},
        date={1996},
     journal={Rev. Mat. Iberoamericana},
      volume={12},
      number={1},
       pages={47\ndash 61},
}

\bib{mcc90c}{article}{
      author={M{\raise.45ex\hbox{c}C}arthy, J.E.},
       title={Common range of co-analytic {T}oeplitz operators},
        date={1990},
     journal={J. Amer. Math. Soc.},
      volume={3},
      number={4},
       pages={793\ndash 799},
}

\bib{mcc92a}{article}{
      author={M{\raise.45ex\hbox{c}C}arthy, J.E.},
       title={Topologies on the {Smirnov} class},
        date={1992},
     journal={J. Funct. Anal.},
      volume={104},
      number={1},
       pages={229\ndash 241},
}

\bib{naw89}{article}{
      author={Nawrocki, M.},
       title={Linear functionals on the {S}mirnov class of the unit ball in
  {$C^n$}},
        date={1989},
     journal={Annales Acad. Sci. Fenn.},
      volume={14},
       pages={369\ndash 379},
}

\bib{PP11}{article}{
      author={Pau, Jordi},
      author={Pel\'{a}ez, Jos\'{e}~\'{A}ngel},
       title={On the zeros of functions in {D}irichlet-type spaces},
        date={2011},
        ISSN={0002-9947},
     journal={Trans. Amer. Math. Soc.},
      volume={363},
      number={4},
       pages={1981\ndash 2002},
      review={\MR{2746672}},
}

\bib{prig}{book}{
      author={Priwalow, I.~I.},
       title={Randeigenschaften analytischer {F}unktionen},
      series={Zweite, unter Redaktion von A. I. Markuschewitsch
  \"{u}berarbeitete und erg\"{a}nzte Auflage. Hochschulb\"{u}cher f\"{u}r
  Mathematik, Bd. 25},
   publisher={VEB Deutscher Verlag der Wissenschaften, Berlin},
        date={1956},
      review={\MR{0083565}},
}

\bib{rud80}{book}{
      author={Rudin, Walter},
       title={{Function Theory in the unit ball of $C^n$}},
   publisher={Springer-Verlag},
     address={Berlin},
        date={1980},
}

\bib{rud86}{book}{
      author={Rudin, Walter},
       title={New constructions of functions holomorphic in the unit ball of
  $c^n$},
      series={C.B.M.S.\ No.~63},
   publisher={American Mathematical Society},
     address={Providence},
        date={1986},
}

\bib{Rudin91}{book}{
      author={Rudin, Walter},
       title={Functional analysis},
     edition={Second},
      series={International Series in Pure and Applied Mathematics},
   publisher={McGraw-Hill Inc.},
     address={New York},
        date={1991},
        ISBN={0-07-054236-8},
      review={\MR{1157815 (92k:46001)}},
}

\bib{sar94}{book}{
      author={Sarason, Donald},
       title={{Sub-Hardy Hilbert spaces in the unit disk}},
   publisher={University of Arkansas Lecture Notes, Wiley},
     address={New York},
        date={1994},
}

\bib{Shalit13}{incollection}{
      author={Shalit, Orr},
       title={Operator theory and function theory in {D}rury-{A}rveson space
  and its quotients},
        date={2015},
   booktitle={Operator theory},
      editor={Alpay, Daniel},
   publisher={Springer},
       pages={1125\ndash 1180},
}

\bib{sha-shi}{article}{
      author={Shapiro, J.H.},
      author={Shields, A.L.},
       title={Unusual topological properties of the {N}evanlinna class},
        date={1973},
     journal={Amer. J. Math.},
      volume={97},
      number={4},
       pages={915\ndash 936},
}

\bib{wil}{book}{
      author={Wilansky, Albert},
       title={Modern methods in topological vector spaces},
   publisher={McGraw-Hill International Book Co., New York},
        date={1978},
        ISBN={0-07-070180-6},
      review={\MR{518316}},
}

\bib{yan73b}{article}{
      author={Yanagihara, N.},
       title={The containing {Fr\'echet} space for the class {$N^+$}},
        date={1973},
      volume={40},
       pages={93\ndash 103},
}

\bib{yan73a}{article}{
      author={Yanagihara, N.},
       title={Multipliers and linear functionals for the class {$N^+$}},
        date={1973},
     journal={Trans. Amer. Math. Soc.},
      volume={180},
       pages={449\ndash 461},
}

\bib{yan74}{article}{
      author={Yanagihara, N.},
       title={Mean growth and {Taylor} coefficients of some classes of
  functions},
        date={1974},
     journal={Ann. Polon. Math.},
      volume={30},
       pages={37\ndash 48},
}

\end{biblist}
\end{bibdiv}

\end{document}